\documentclass[12pt,a4paper,reqno]{amsart}
\usepackage{amsfonts,amsthm,amsmath,amscd}
\usepackage[utf8]{inputenc}
\usepackage{t1enc}
\usepackage[mathscr]{eucal}
\usepackage{indentfirst}
\usepackage{tikz}
\usepackage{caption}
\usepackage{subcaption}
\usepackage[font=small,labelfont=bf]{caption}
\usetikzlibrary{automata,arrows,positioning,calc}
\usepackage{graphicx,graphics,pict2e}
\usepackage{tkz-berge}
\usepackage{float}
\usepackage{epic}
\usepackage{lipsum}
\usepackage{stmaryrd}
\usepackage{authblk}
\usepackage{subcaption}
\usepackage[symbol]{footmisc}
\numberwithin{equation}{section}
\usepackage[margin=2.9cm]{geometry}
\usepackage{epstopdf}
\RequirePackage{doi}
\usepackage{hyperref}
\allowdisplaybreaks
\usepackage{cite}

\usepackage{verbatim,amssymb,amsfonts}
\usepackage{mathrsfs}
\usepackage{mathtools}
\DeclarePairedDelimiter\floor{\lfloor}{\rfloor}
\usepackage{tikz-cd}
\usepackage{indentfirst}
\usepackage{slashed}
\usepackage{thmtools}
\usepackage{setspace}
\usepackage{enumerate}
\usepackage{accents}

\theoremstyle{plain}
\newtheorem{theorem}{Theorem}[section]
\newtheorem{lemma}[theorem]{Lemma}
\newtheorem{corollary}[theorem]{Corollary}
\newtheorem{proposition}[theorem]{Proposition}

 \theoremstyle{definition}

\newtheorem{remark}[theorem]{Remark}

\DeclarePairedDelimiterX{\inp}[2]{\langle}{\rangle}{#1, #2}

\newcommand{\cV}{{\mathcal V}}

\newcommand{\rank}{{\mathrm{rank}}}

\newcommand{\dist}{{\mathrm{dist}}}

\newcommand{\ba}{\begin{eqnarray}}
\newcommand{\na}{\end{eqnarray}}
\newcommand{\ban}{\begin{eqnarray*}}
\newcommand{\nan}{\end{eqnarray*}}

\newcommand{\N}{{\mathbb N}}

\newcommand{\R}{{\mathbb R}}

\newcommand{\scrC}{{\mathscr C}}

\renewcommand{\thefootnote}{\fnsymbol{footnote}}

\makeatletter
\g@addto@macro{\endabstract}{\@setabstract}
\newcommand{\authorfootnotes}{\renewcommand\thefootnote{\@fnsymbol\c@footnote}}%
\makeatother

\makeatletter
\@namedef{subjclassname@2020}{%
  \textup{2020} Mathematics Subject Classification}
\makeatother

\title[]{A Characterization of triangle-free \\ cyclic graphs with self-loops of rank 3}



\begin{document}

\begin{center}
    \vspace{-1cm}
	\maketitle
	
	\normalsize
    \authorfootnotes
    Johnny Lim
	\par \bigskip



        \small{School of Mathematical Sciences, Universiti Sains Malaysia, Penang, Malaysia}\par \bigskip

\end{center}

\address{School of Mathematical Sciences, Universiti Sains Malaysia, Penang, Malaysia
}
\email{johnny.lim@usm.my}

\begin{abstract}
Let $G_S$ be a self-loop graph as the graph obtained by attaching a self-loop at every vertex in $S \subseteq V(G)$ of a simple graph $G.$ If $G=C_n$ is the cycle graphs of order $n$ and $S \neq \emptyset,$ we show that there are no rank 3 self-loop graphs $(C_n)_S$ for $n\geq 5.$ As a consequence, we determine and construct all possible rank 3 triangle-free self-loop cyclic graph of order at least 4 from $(C_4)_S$ via graph join operations. This provides a partial solution to the characterization problem of rank 3 self-loop graphs.  
\end{abstract}

\section{Introduction}
Let $G$ be a simple (i.e., without self-loops and multiple edges) undirected graph of order $n=|V(G)|,$ where $V(G)=\{v_1,v_2,...,v_n\}$ is the set of vertices in $G.$ In this paper, we consider only \textit{connected} graphs, for which there exists a path between any two vertices. Let $A(G)=(a_{ij})_{n\times n}$ be the adjacency matrix of $G,$ whose the $(i,j)$-entry is defined by $a_{ij}=1$ if $v_i$ is adjacent to $v_j$ where $i \neq j$ and $a_{ij}=0$ otherwise. For a matrix $A \in \mathrm{Mat}_n(\R),$ the \textit{rank} of $A,$ $\rank(A),$ is the dimension of the column space of $A.$ Thus, the rank of a graph $G,$  $\rank(G),$ is defined to be $\rank(A(G)).$ 

Let $S \subseteq V(G)$ be a non-empty subset and $|S|=\sigma.$ Denote by $G_S$ the \textit{self-loop graph} of $G$ at $S,$ as a graph obtained from $G$ by attaching a self-loop at each vertex in $S.$  If $\sigma = 0,$ then $G_S=G.$ If $\sigma=n,$ we denote the graph with full-loops by $\widehat{G}.$ The adjacency matrix $A(G_S)=(a_{ij})_{n \times n}$ of $G_S$ is then defined by $a_{ij}=1$ if $v_i$ is adjacent to $v_j$ with possibly $i=j,$ and $a_{ij}=0$ otherwise. Thus, $A(G_S)=A(G) + I_\sigma,$ where $I_\sigma$ is the diagonal matrix with $\sigma$ many 1's on the diagonal and zero otherwise. The rank of $G_S$ is defined similarly by $\rank(G_S):= \rank(A(G_S)).$ Throughout this paper, we work primarily with the rank of graphs. In particular, it is classically known that:
\begin{proposition}\cite{HornJohnson2013}
	If $A$ is a submatrix of $B,$ then $\rank(A) \leq \rank(B).$
\end{proposition}
\noindent An immediate consequence is the following that will be used frequently henceforth: 
\begin{corollary}
	\label{subgraphrank}
	If $H$ is a subgraph of $G_S,$ then $\rank(H) \leq \rank(G_S).$
\end{corollary}

Throughout this paper, we adopt the standard convention by denoting $K_n,$ $P_n,$ and $C_n$ to be the \textit{complete graph}, \textit{path graph}, and \textit{cycle graph} of order $n,$ respectively. For our purpose, $(P_n)_S$ will be the self-loop path graph of order $n,$ and thus denote
\[
(P_n)_{\alpha_1\alpha_2 \cdots \alpha_i}, \quad 1 \leq \alpha_1 \leq \cdots \leq \alpha_i \leq \cdots \leq n,
\]
to represent each loop located at vertices $\alpha_1,\ldots, \alpha_i,$ respectively.\footnote{The notational consistency up to isomorphism is implicitly assumed. For example in Fig.~\ref{4P1} and \ref{4P2} one may interpret the right-most vertex as the starting point, thus $(P_4)_1 \cong (P_4)_4$ and $(P_4)_3 \cong (P_4)_2.$}
For example,

\vspace{-1em}
\begin{figure}[H]
	\centering
	\begin{minipage}[b]{0.4\textwidth}
		\centering    \includegraphics[width=0.75\textwidth]{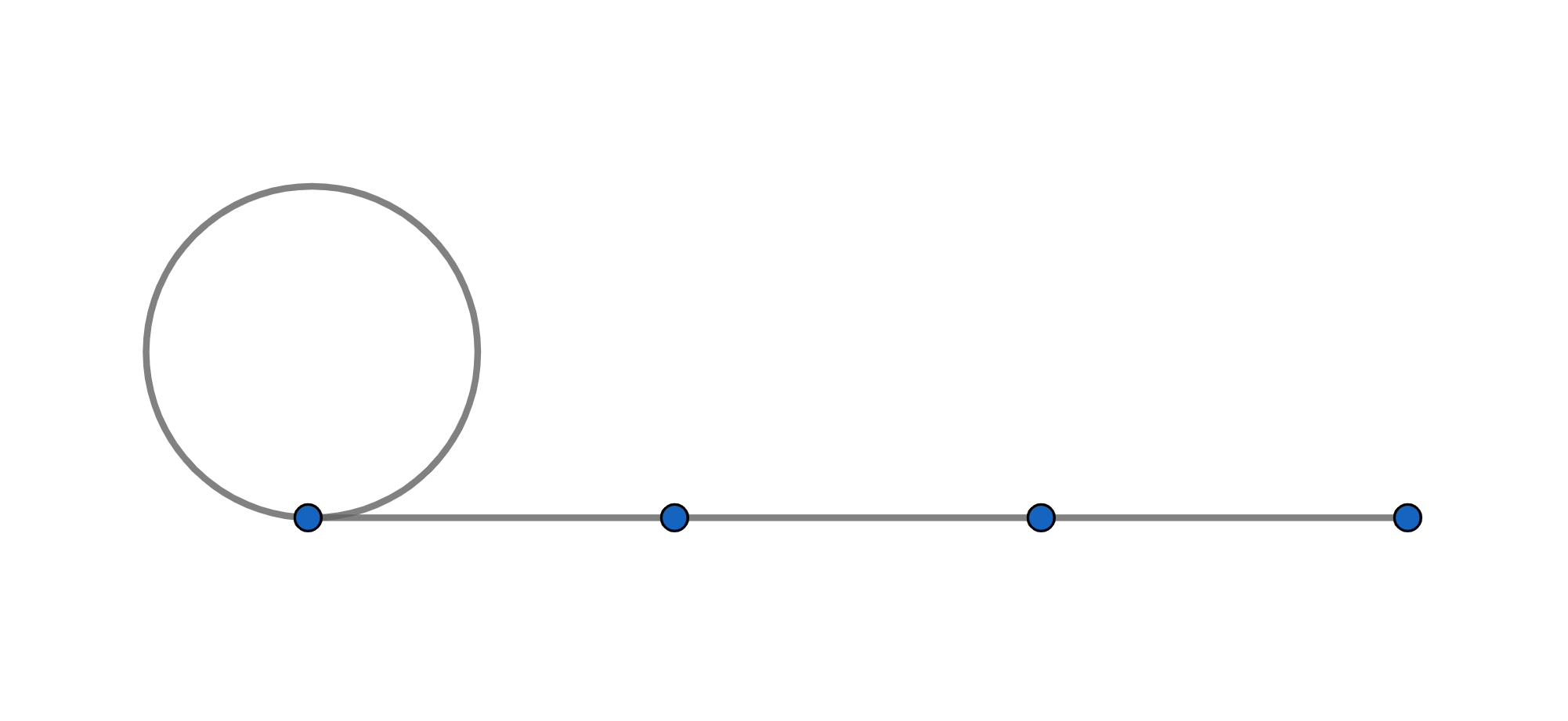}
		\caption{$(P_4)_1$}
		\label{4P1}
	\end{minipage}
	\hspace{2em}
	\begin{minipage}[b]{0.4\textwidth}
		\centering    \includegraphics[width=0.65\textwidth]{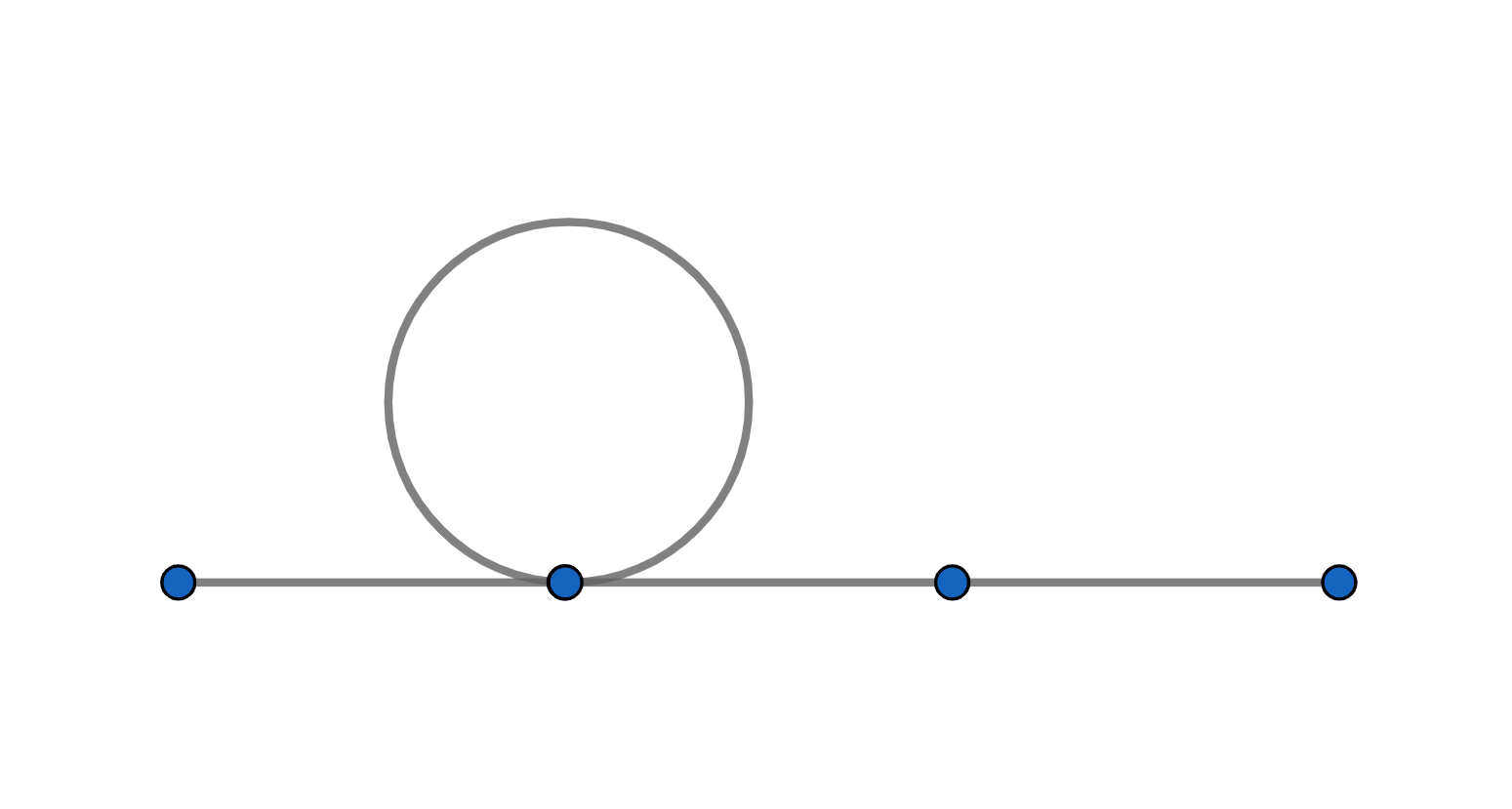}
		\caption{$(P_4)_2$}
		\label{4P2}
	\end{minipage}
\end{figure}

In some proofs, we utilize the following two notions (cf. \cite{cvetkovic2010intro}): the first is the \textit{distance} between vertices. For vertices $v_i$ and $v_j,$ denote by $\dist(v_i,v_j)$ to be the length (i.e., number of edges) of a shortest path between $v_i$ and $v_j.$ If $U$ and $U'$ are non-empty subsets $V(G)$, then the distance $\dist(U,U')$ is defined to be the minimum distance between all vertices $U$ and $U',$ i.e., 
\[
\dist(U,U'):= \min\{\dist(v,v') \mid v \in U, v' \in U'\}.
\]
The second is the graph joins. The \textit{join} of (disjoint) graphs $G_1$ and $G_2,$ denoted by $G_1 \vee G_2,$ is the graph obtained by from the union of disjoint copies of $G_1$ and $G_2$ by joining each vertex of $G_1$ to each vertex of $G_2.$ For our purpose, if $A\subseteq V(G_1),$ we write $G_1 \underset{A}{\vee} G_2$ to mean the graph join of $G_1$ and $G_2$ \textit{over} the set $A,$ that is for which each vertex of $A$ is joined with each vertex of $G_2$ and none of $V(G_1)\backslash A$ is joined with vertices of $G_2.$ 

The study of spectral properties and energy of self-loop graphs is very recent. It was initiated by Gutman et. al.  \cite{gutman2021energy} in 2021. From the perspective of Molecular Chemistry, self-loop graphs represent heteroatoms and its graph energy is closely related to the total $\pi$-electron energy of conjugated molecules, cf. \cite{gutman1990topological} and references therein.
In 2023, the author and team \cite{akbari2023selfloop} have explored some spectral results on self-loop graphs, as well as a characterization of self-loop graphs of order $n$ whose eigenvalues are all positive or non-negative, and for any graphs with a few distinct eigenvalues. On the other hand, characterizations of simple graphs in terms of low rank are well-studied in the literature. In particular, Sciriha \cite[Sect.6]{sciriha1999} proved the characterization of simple graphs of rank 0, 1, 2, and 3: 
\begin{enumerate}[(i)]
	\item A simple graph is of rank 0 if and only if it is an empty graph. 
	\item There are no simple graphs of rank 1. 
	\item The only simple graphs of order $n$ and rank 2 are the complete bipartite graphs $K_{r,n-r},$ for $1 \leq r \leq \floor{\frac{n}{2}}.$
	\item The only simple graphs of rank 3 are the complete tripartite graphs $K_{a,b,c}.$
\end{enumerate}
Characterization results for rank 4 and 5 were then obtained by \cite{chang2011} and \cite{chang2012}, respectively. It is then natural to ask the question: \textit{what about the characterization of self-loop graphs with low rank}? This motivated the study on characterization of self-loop graphs in \cite{rank12char2025}, for which those with rank 1 and 2 are completely characterized. 

\begin{theorem}\cite{rank12char2025}
	Let $G_S$ be a connected self-loop graph of order $n.$ Then, 
	\begin{enumerate}[(i)]
		\item $G_S$ is of rank 1 if and only if $G_S \cong \widehat{K_n}.$
		
		\item $G_S$ is of rank 2 if and only if $G_S \cong \widehat{K_\sigma} \:\vee\: (n-\sigma)K_1$ for $1 \leq \sigma \leq n-1.$  
	\end{enumerate}
\end{theorem}
Some results on rank 3 characterizations of connected self-loop graphs have been obtained and the research is currently in progress. It is the purpose of this paper to establish a \textit{partial} result for rank 3 characterizations. The organization of this paper is as follows. In  Sect.~\ref{Sec2}, we prove that it is not possible to obtain rank 3 graphs that contain the self-loop cycle graphs of order $n\geq 5.$ As a consequence, we provide a characterization of rank 3 triangle-free self-loop cyclic graphs of order \textit{at least} 4 in Sect.~\ref{Sec3}, more precisely they must be isomorphic to one of the graphs in Fig.~\ref{fig:H1}, Fig.~\ref{fig:H2}, Fig.~\ref{fig:H2'}, Fig.~\ref{fig:H3}, Fig.~\ref{fig:H4}, and Fig.~\ref{fig:H5}. Lastly, we propose some interesting open problems that are not pursued in this article in Sect.~\ref{openproblem}.

\section{Initial results on self-loop cycle graphs of order at least five}
\label{Sec2}
\subsection{Consecutive loops}

We begin to investigate the possibility of obtaining rank 3 self-loop graphs based on the position of loops in $(C_n)_S$ for $n\geq 5.$ The discussion will be separated into two scenarios, one for which all $\sigma$ loops are consecutive, and another that are not. More specifically, if $\ell_1,\ell_2,\ldots, \ell_\sigma$ are the $\sigma$ loops, then consecutive loops means $\dist(v_{\ell_i},v_{\ell_{i+1}})=1$ for all $1\leq i \leq \sigma-1.$   

\begin{lemma}
	\label{consloop}
	For $n \geq 5,$ there are no rank 3 cycle graphs $(C_n)_S$ with consecutive loops.
\end{lemma}

\begin{proof}
	For $n\geq 5,$ if $1\leq \sigma \leq n-3,$ then $(C_n)_S$ contains the subgraph $(P_4)_1$ (cf. Fig.~\ref{4P1}) of rank 4. Thus, it suffices to consider only consecutive loops with $\sigma \geq n-2.$ It is clear that $(C_5)_S$ with $\sigma=3, 4, 5$ will be of rank $5, 5, 4,$ respectively. For $n\geq 6,$ for $\sigma\geq 4,$ $(C_n)_S$ contains the subgraph $\widehat{P_4}$ (cf. Fig.~\ref{4P1234}) of rank 4. By Corollary~\ref{subgraphrank}, there are no rank 3 cycle graphs of order $n\geq 5$ with consecutive loops. 
\end{proof}

\begin{figure}[H]
	\centering    \includegraphics[width=0.5\textwidth]{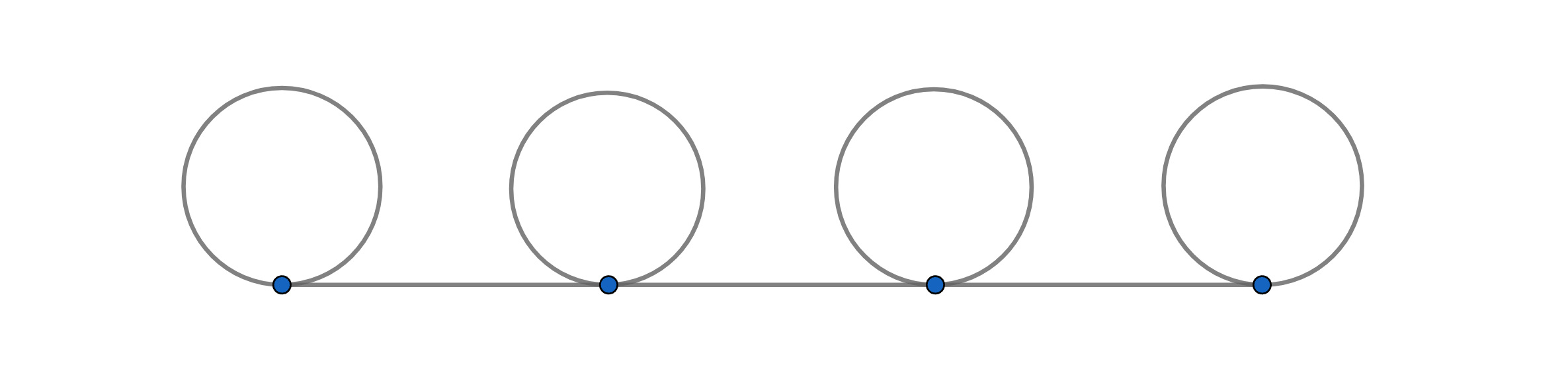}
	\caption{$\widehat{P_4}$}
	\label{4P1234}
\end{figure}
\subsection{Non-consecutive loops}

It follows from the proof of Lemma~\ref{consloop} that $(C_n)_S$ with consecutive loops of $\sigma \geq 4$ will not be of rank 3. Let's consider a \textit{cluster} of loops: $\scrC_i, i=1,2,3,$ to be the clusters of one, two, and three \textit{consecutive} loops, respectively. 
For ``non-consecutive'' to be meaningful, we require $\dist(\scrC_i, \scrC_j) \geq 2$ for $1 \leq i,j\leq 3.$ Thus, it suffices to consider $
S= \bigcup^3_{i=1}a_i \scrC_i$ where $a_i \in \N \cup \{0\}$ is the number of such clusters on $(C_n)_S,$ so that $\sigma=|S|= a_1 + 2a_2 + 3a_3.$ Some examples of different clusters of loops are given below for better visualization.

\vspace{-1em}
\begin{figure}[H]
	\centering
	\hspace{-2em}
	\begin{minipage}[b]{0.35\textwidth}
		\centering    \includegraphics[width=0.9\textwidth]{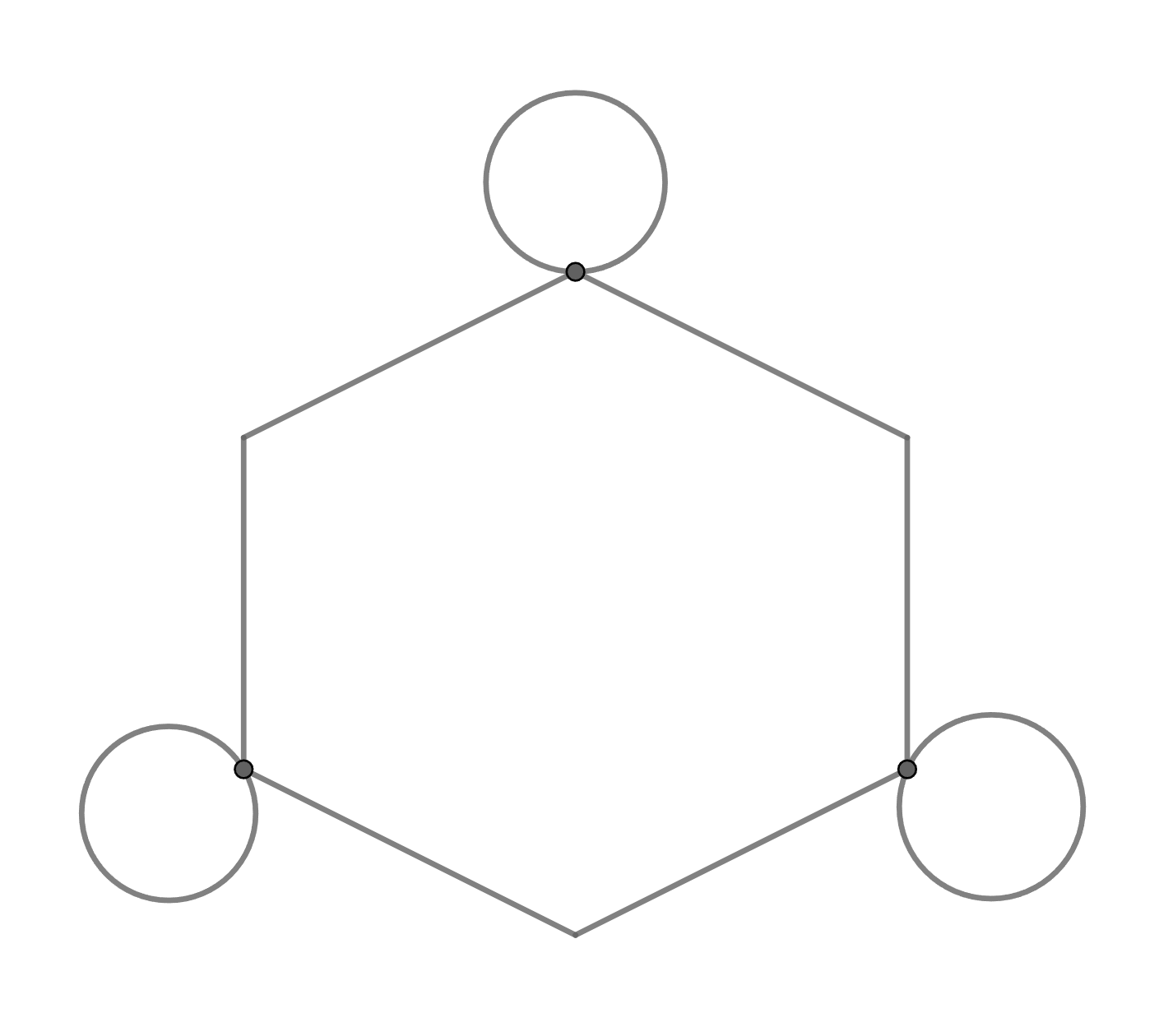}
		\caption{$3\scrC_1$}
	\end{minipage}
	\hspace{-2em}
	\begin{minipage}[b]{0.35\textwidth}
		\centering    \includegraphics[width=0.92\textwidth]{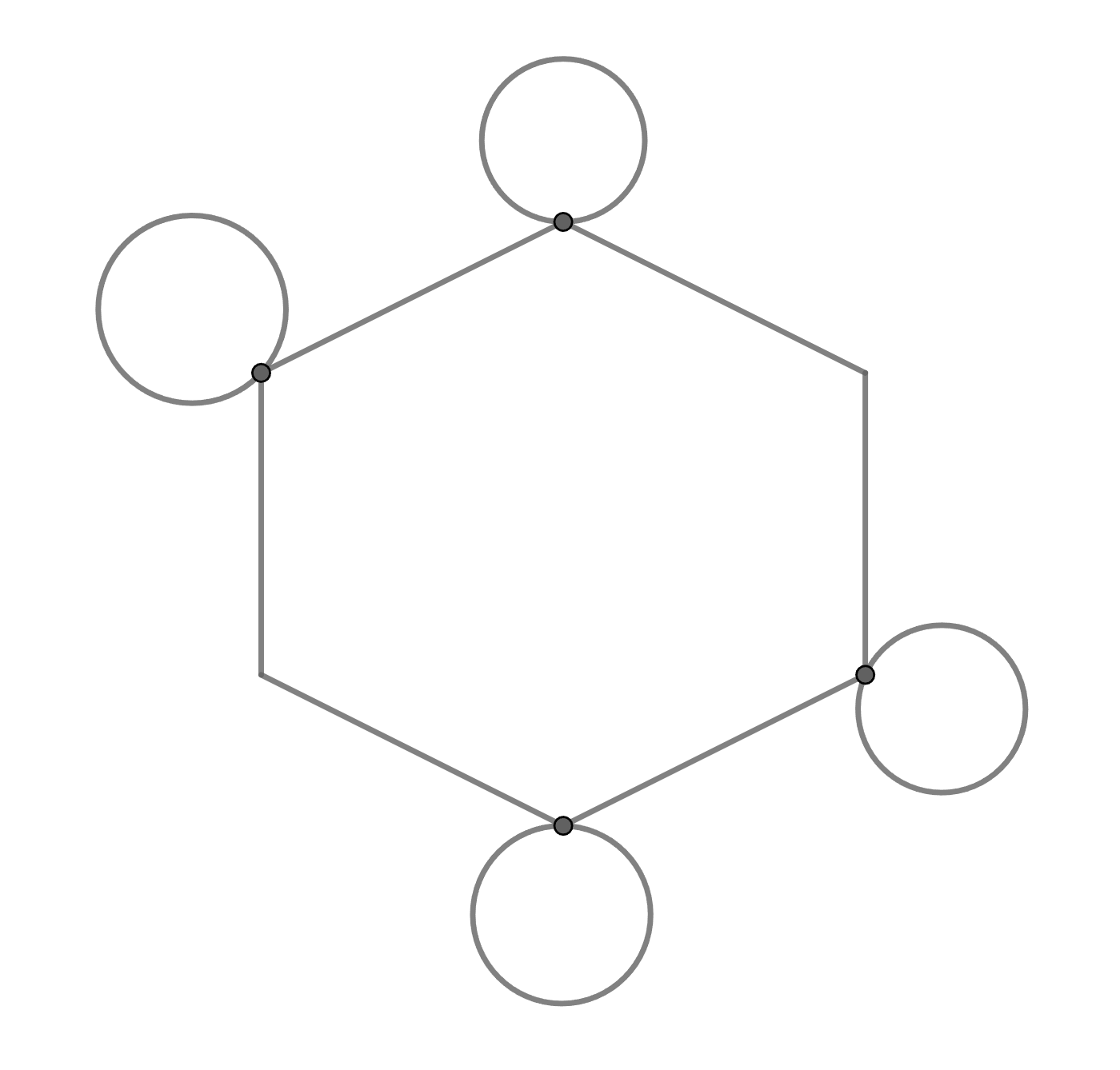}
		\caption{$2\scrC_2$}
	\end{minipage}
	\hspace{-2em}
	\begin{minipage}[b]{0.4\textwidth}
		\centering    \includegraphics[width=0.83\textwidth]{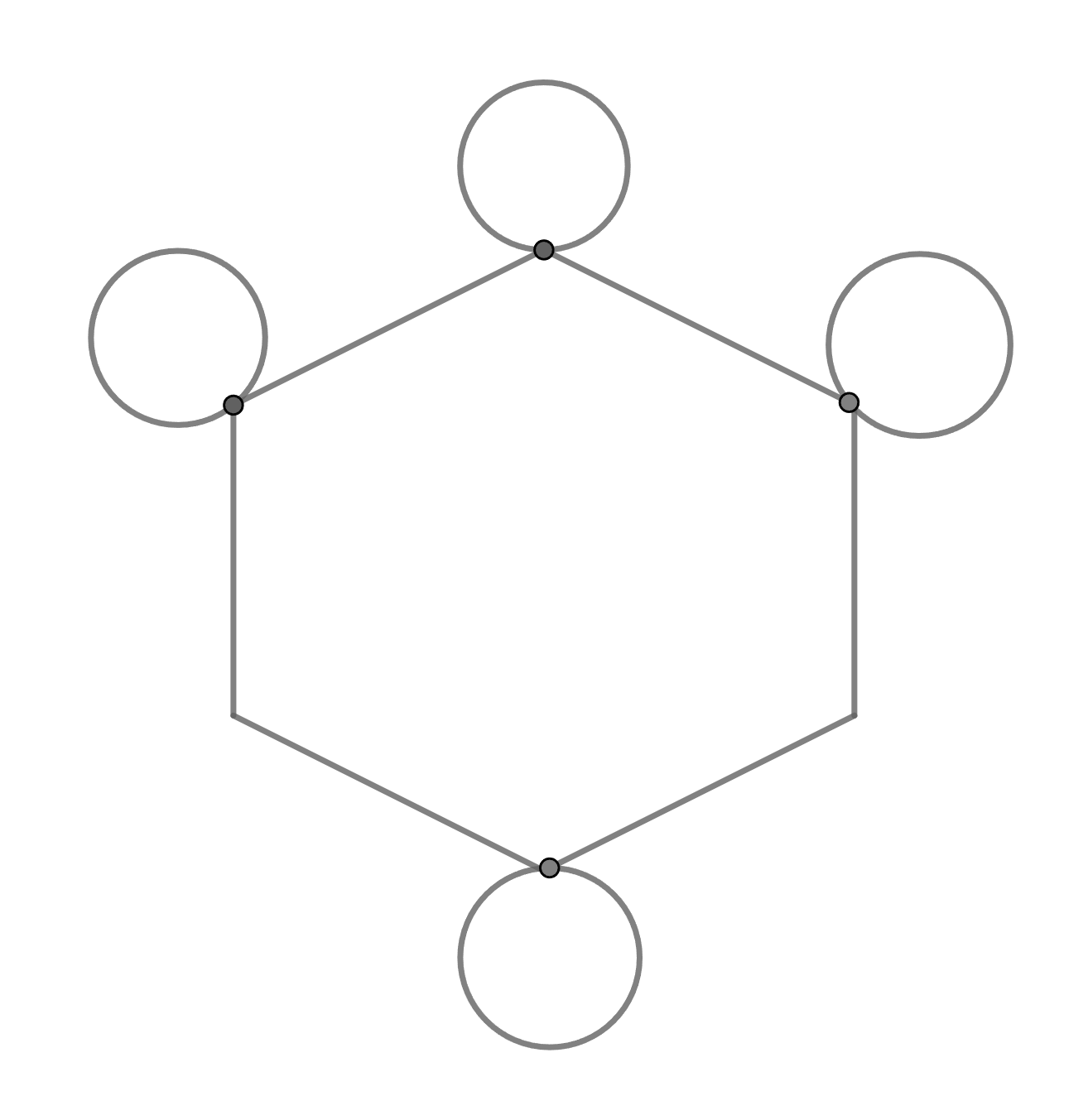}
		\caption{$\scrC_1 \cup \scrC_3$}
	\end{minipage}
\end{figure}

\begin{lemma}
	For $n \geq 5,$ there are no rank 3 cycle graphs $(C_n)_S$ with non-consecutive loops.    
\end{lemma}

\begin{proof}
	We split the proof into discussing several cases of `gaps' between $\scrC_i.$ 
	\begin{enumerate}[(i)]
		\item Let $\dist(\scrC_i,\scrC_j)=2.$ Observe that any pair $(\scrC_i,\scrC_j)$ for $1 \leq i,j \leq 3,$ except for $(\scrC_1,\scrC_1),$ contains the subgraph $(P_4)_{124}$ of rank 4.
		\begin{figure}[H]
			\centering    \includegraphics[width=0.4\textwidth]{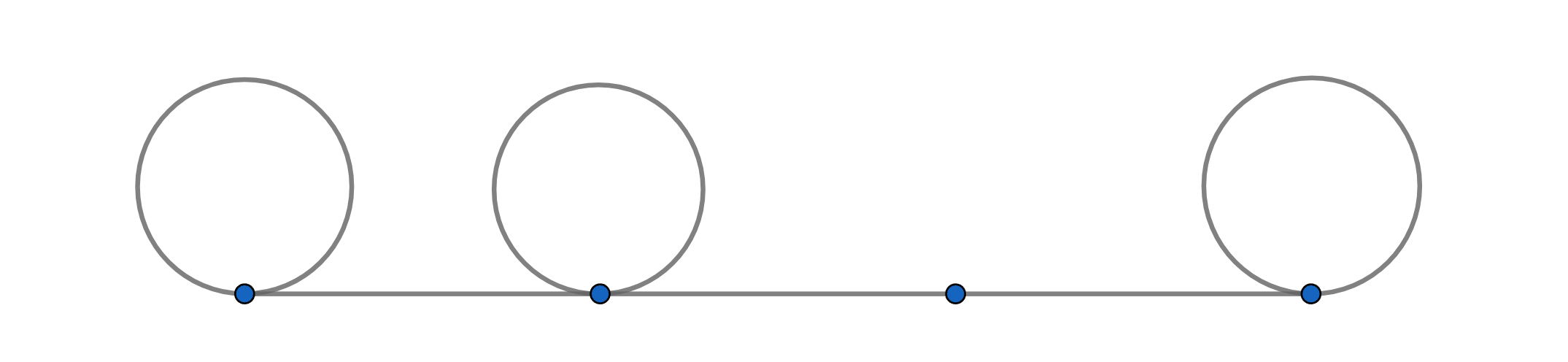}
			\caption{$(P_4)_{124}$}
			\label{4P124}
		\end{figure}
		\noindent For $(\scrC_1,\scrC_1),$ since $n \geq 5,$ there exists a vertex $v_0$ such that $(C_n)_S$ contains the subgraph $(P_4)_{13}$ of rank 4.
		\begin{figure}[H]
			\centering    \includegraphics[width=0.4\textwidth]{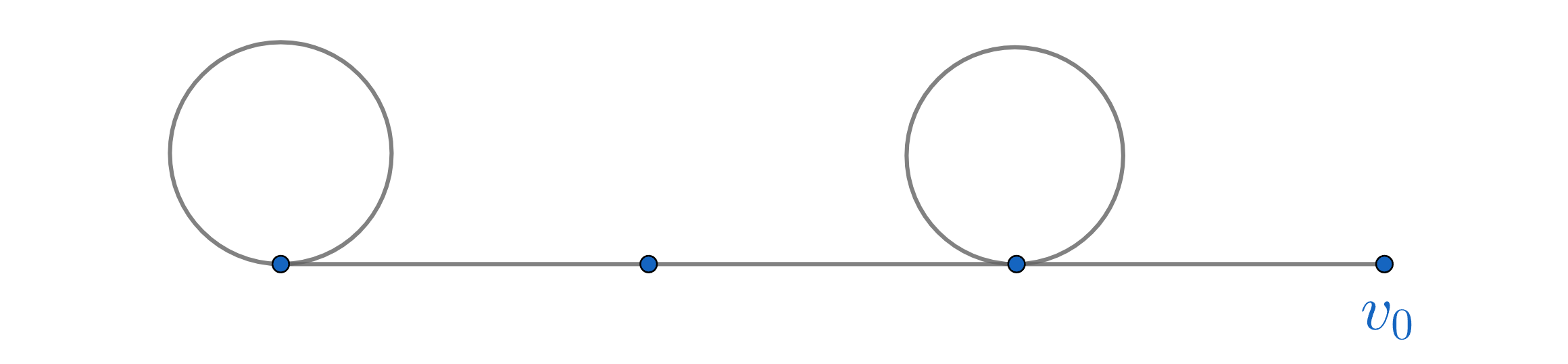}
			\caption{$(P_4)_{13}$}
			\label{4P13}
		\end{figure}
		\item Let $\dist(\scrC_i,\scrC_j)=3.$ Then any pair $(\scrC_i,\scrC_j)$ for $1 \leq i,j \leq 3,$ except for $(\scrC_1,\scrC_1),$ contains the subgraph $(P_5)_{145}$ of rank 5. 
		\begin{figure}[H]
			\centering    \includegraphics[width=0.55\textwidth]{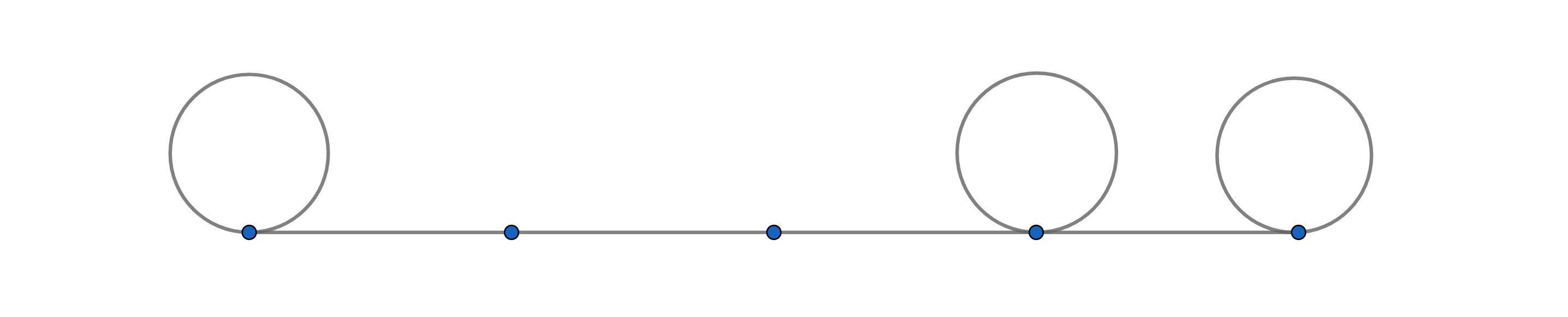}
			\caption{$(P_5)_{145}$}
			\label{5P145}
		\end{figure}
		\noindent For $(\scrC_1,\scrC_1)$ and $n \geq 6,$ $(C_n)_S$ contains the subgraph $(P_5)_{14}$ of rank 5;
		\begin{figure}[H]
			\centering    \includegraphics[width=0.5\textwidth]{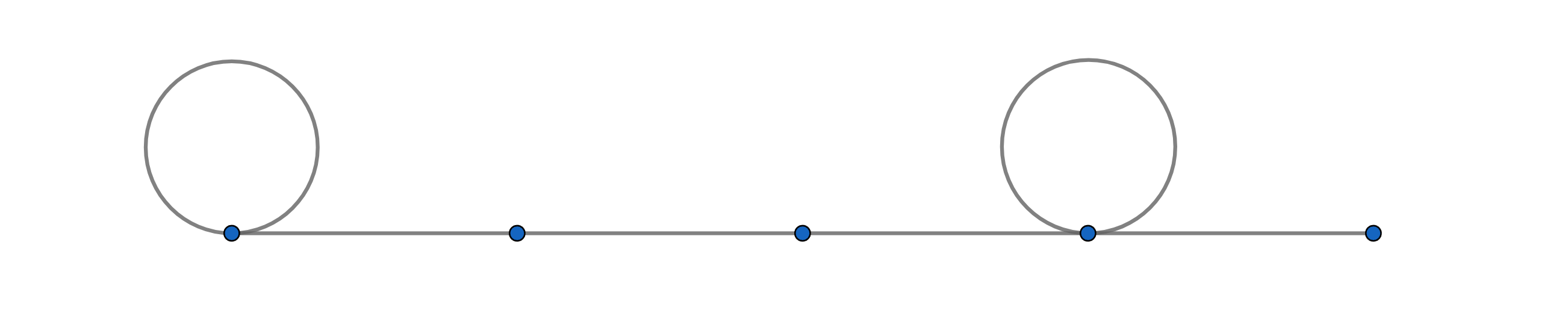}
			\caption{$(P_5)_{14}$}
			\label{5P14}
		\end{figure}
		\noindent if $n=5,$ the graph contains the subgraph $(P_4)_{13}$ of rank 4, similar to Fig.~\ref{4P13}.
		\item Let $\dist(\scrC_i,\scrC_j) \geq 4.$ Then, $(C_n)_S$ contains the subgraph $(P_4)_1$ (cf. Fig.~\ref{4P1}) of rank 4.
	\end{enumerate}
	By applying Corollary~\ref{subgraphrank}, we reach at the desired claim.
\end{proof}

These two lemmas can be concluded in one theorem.
\begin{theorem}
	Let $(C_n)_S$ be the cycle graph of order $n\geq 5.$ For any nonempty $S\subseteq V(C_n),$ then $\rank((C_n)_S) \geq 4.$      
\end{theorem}
\section{Main results}
\label{Sec3}

From the results established above, one ought to consider cycles graphs of low order to possibly obtain rank 3 self-loop graphs. In fact, considerable effort has been put forth in \cite{rank12char2025} where rank 1 and rank 2 connected self-loop graphs are completely characterized. Here, we focus on getting partial results for rank 3 connected self-loop graphs by working with 4-cycle graphs.

\begin{theorem}
	\label{thm1}
	Up to isomorphisms, all rank 3 triangle-free connected cyclic graphs with one self-loop constructed from $G_S=(C_4)_S$ with $|S|=1$ are 
	\[
	H_1= G_S \underset{\cV_I}{\vee} W \quad\text{ and }\quad 
	H_2= G_S \underset{V(G_S)\setminus\cV_I}{\vee} W,
	\]
	where $\cV_{I}$ is the independent set of loopless vertices in $V(G_S),$ and $W \nsubseteq V(G_S)$ is an independent set of $K_1'$s.
	
\end{theorem}

\begin{proof}
	Let $\cV_{I}= \{v_a,v_b\}.$ If $v \in W,$ then $v \neq v_a,v_b.$ If $v$ is adjacent to $v_a$ (or $v_b$) and at least one vertex in $N(a)$ (or $N(b)$), then at least one triangle is formed. Thus, $v$ is adjacent to either 
	
	\begin{center}
		\begin{enumerate}[(i)]
			\item only one vertex of $G_S,$ or
			\item only $v_a$ and $v_b,$ or 
			\item only $N(v_a)$ and $N(v_b).$
		\end{enumerate}
	\end{center}
	For (i), the resulting graph is either of the following three graphs:
	\begin{figure}[H]
		\centering
		\hspace{-2em}    
		\vspace{-1em}
		\begin{minipage}[b]{0.33\textwidth}
			\centering    \includegraphics[width=0.7\textwidth]{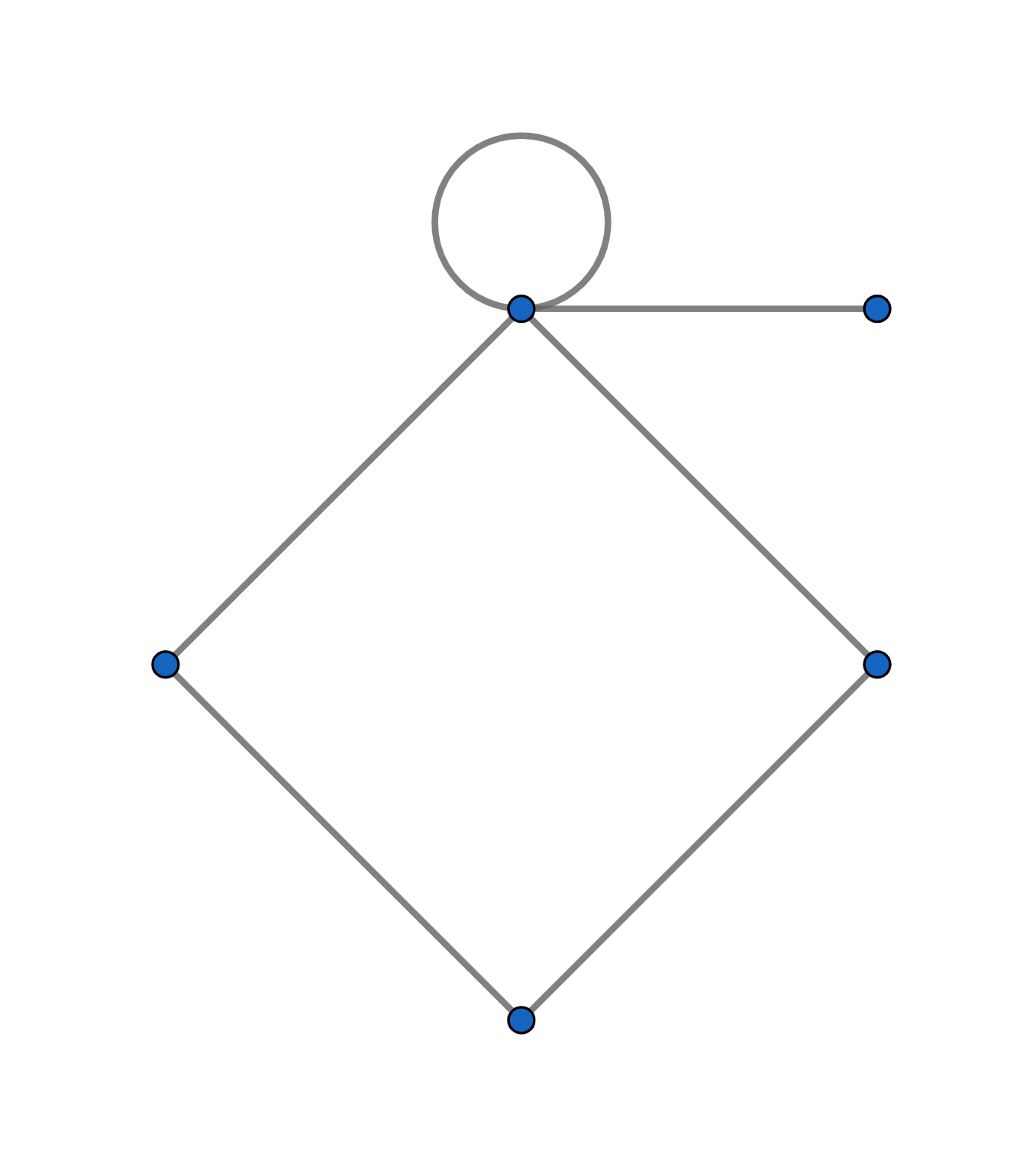}
			\caption{}
		\end{minipage}
		\hspace{-1em}
		\begin{minipage}[b]{0.33\textwidth}
			\centering    \includegraphics[width=1.0\textwidth]{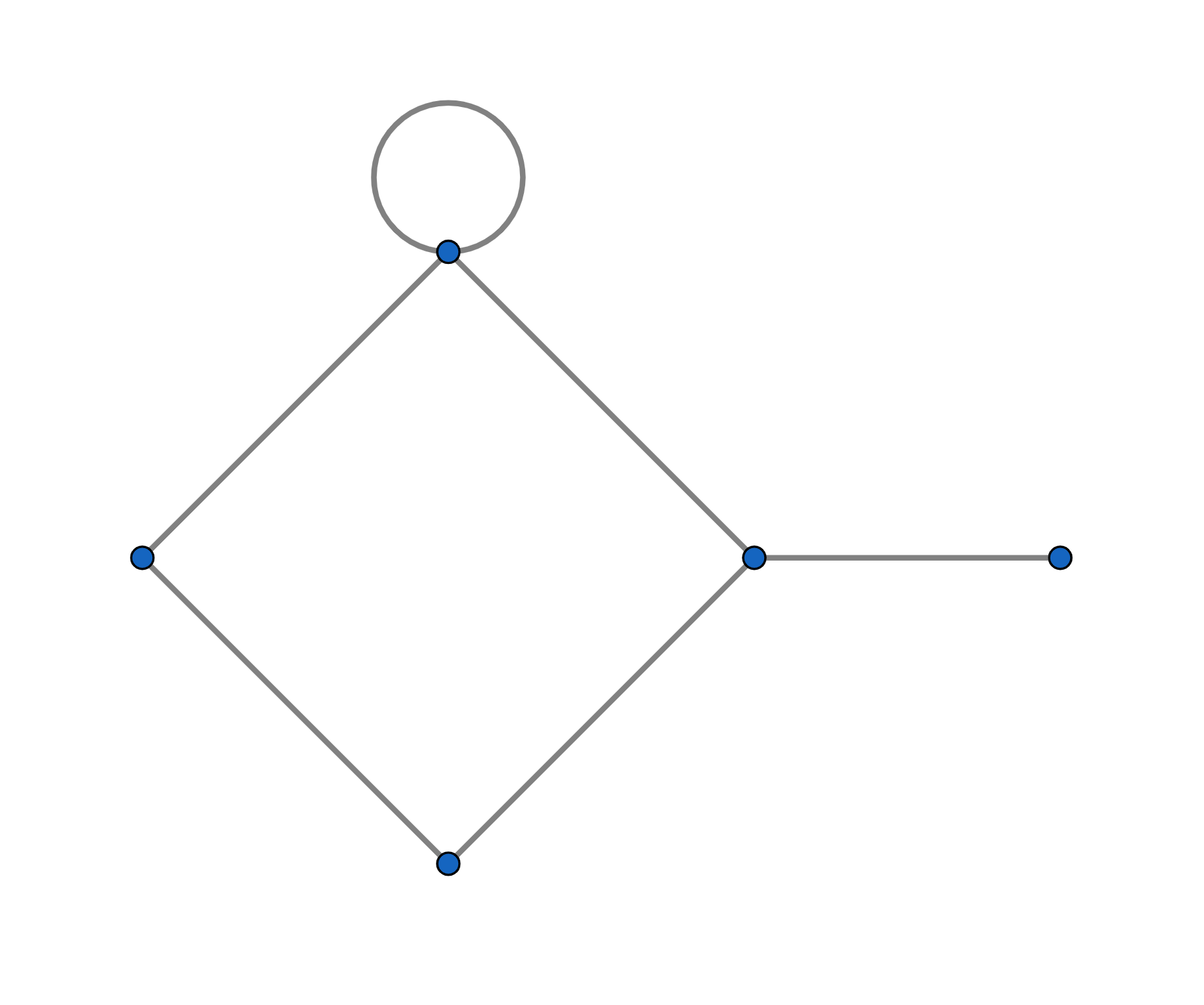}
			\caption{}
		\end{minipage}
		\hspace{-1em}
		\begin{minipage}[b]{0.33\textwidth}
			\centering    \includegraphics[width=0.75\textwidth]{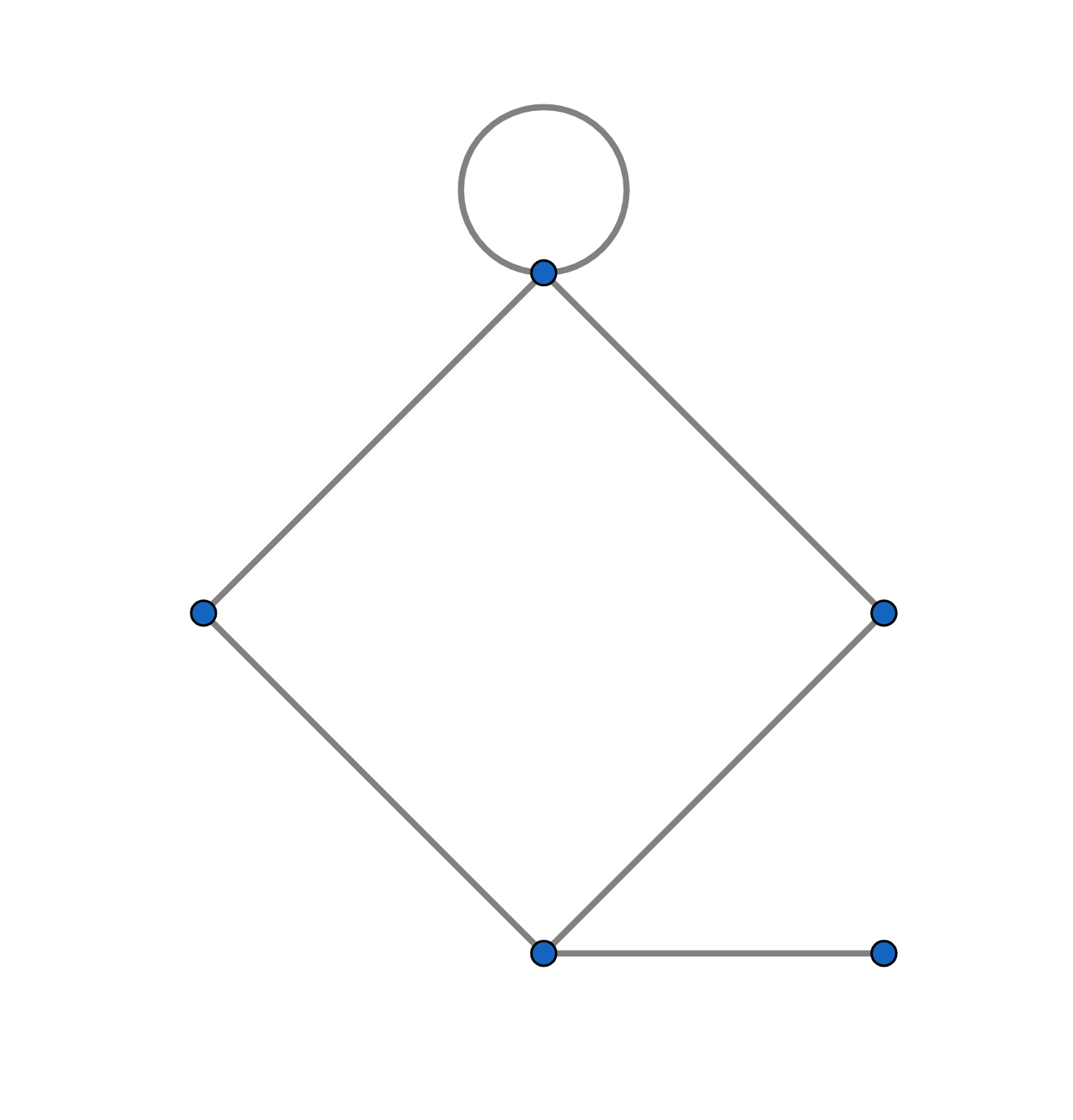}
			\caption{}
		\end{minipage}
	\end{figure}
	\noindent all of which contain either of the two subgraphs $(P_4)_1$ and $(P_4)_2$ of rank 4, so case (i) can be ignored.
	The remaining cases (ii) and (iii) give two non-isomorphic graphs of rank 3:
	\begin{figure}[H]
		\centering
		\hspace{-1em}
		\begin{minipage}[b]{0.4\textwidth}
			\centering    
			\includegraphics[width=0.9\textwidth]{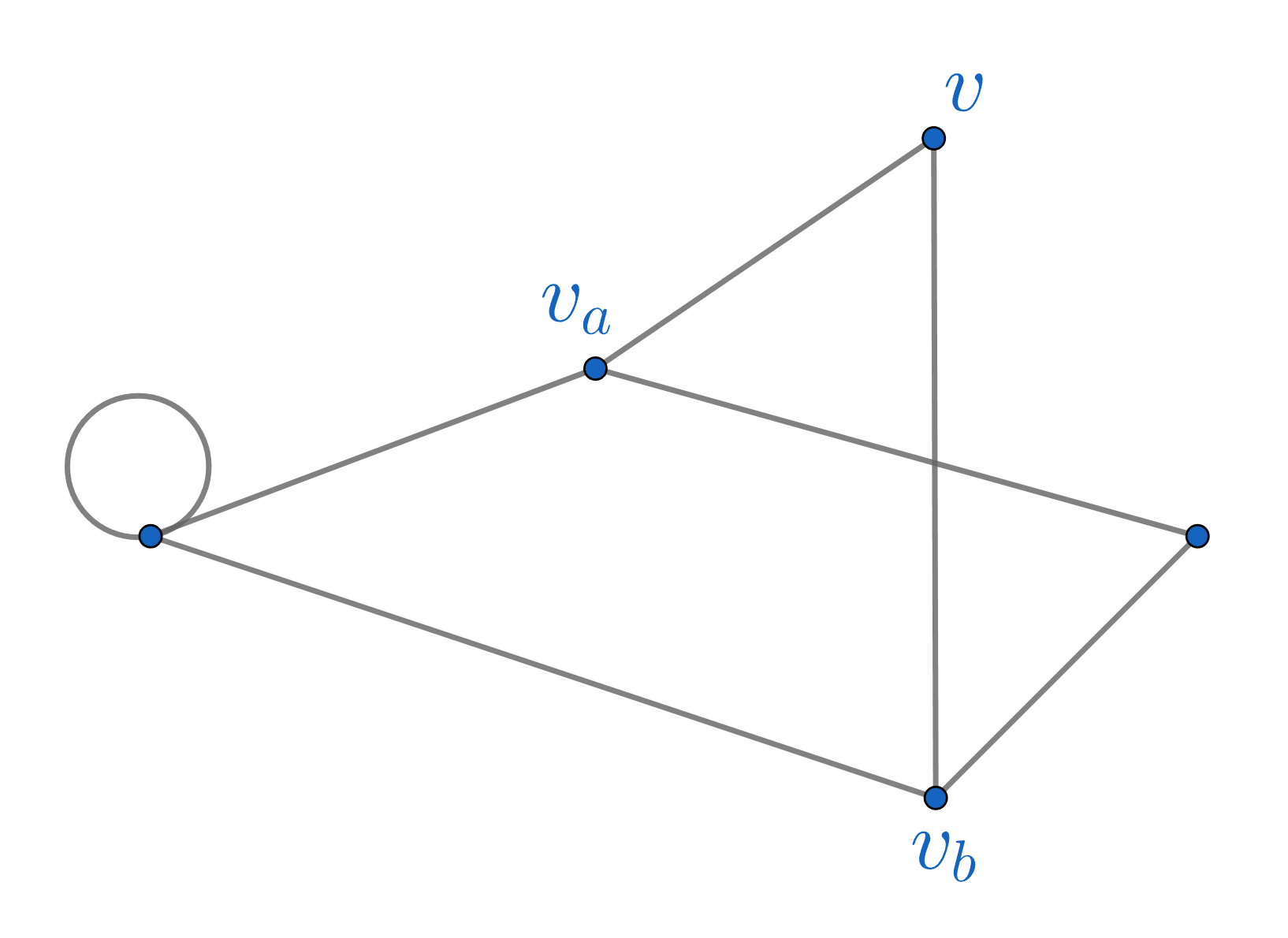}
			\caption{}
			\label{fig:14}
		\end{minipage}
		\hspace{3em}
		\begin{minipage}[b]{0.4\textwidth}
			\centering    
			\includegraphics[width=0.9\textwidth]{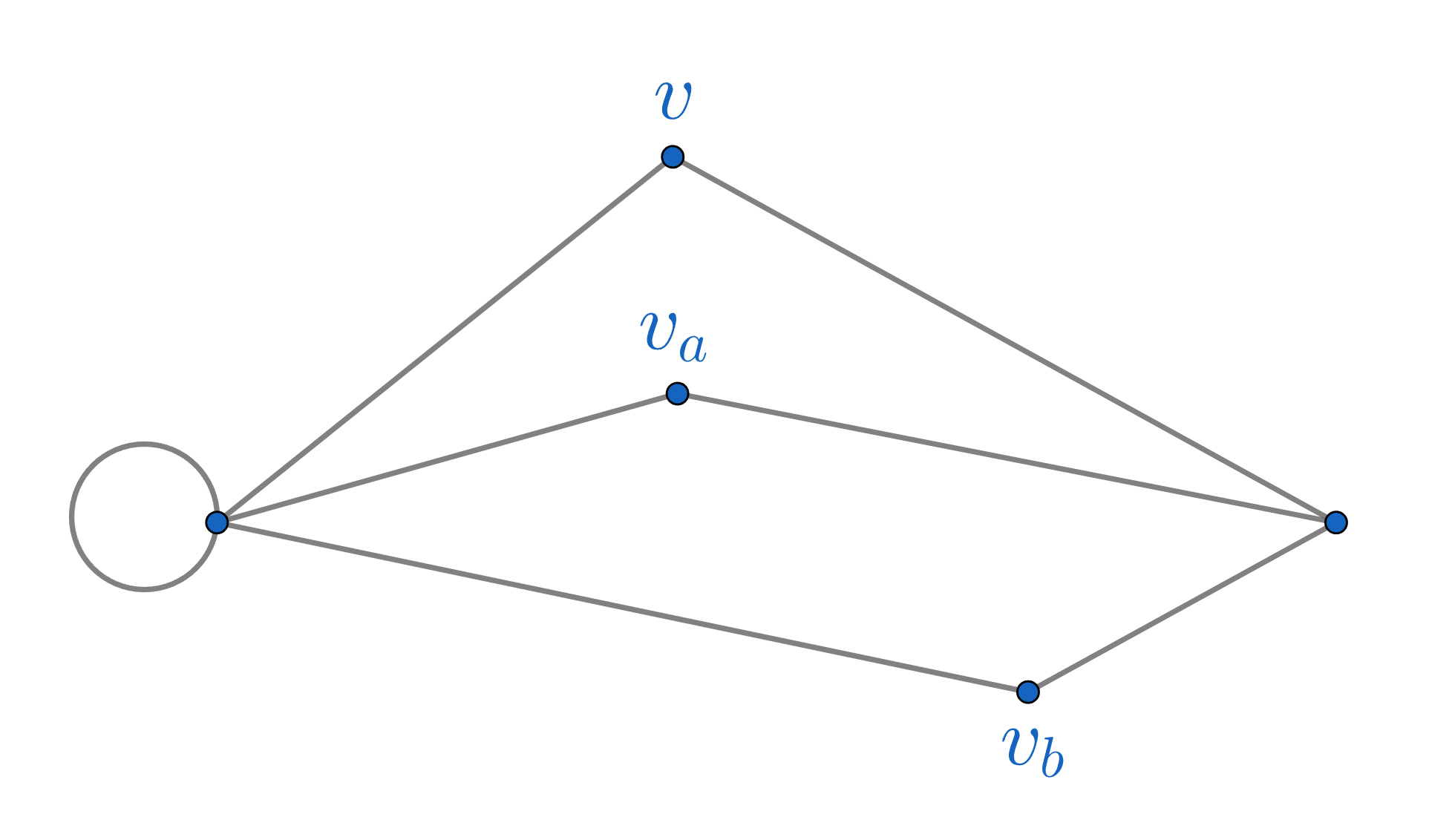}
			\caption{}
			\label{fig:15}
		\end{minipage}
	\end{figure}
	By induction on the graph joins, we obtain the following graphs of rank 3, which we refer as $H_1$ and $H_2,$ respectively
	\begin{figure}[H]
		\centering
		\hspace{-2em}
		\vspace{-5em}
		\begin{minipage}[b]{0.4\textwidth}
			\centering    
			\includegraphics[width=0.9\textwidth]{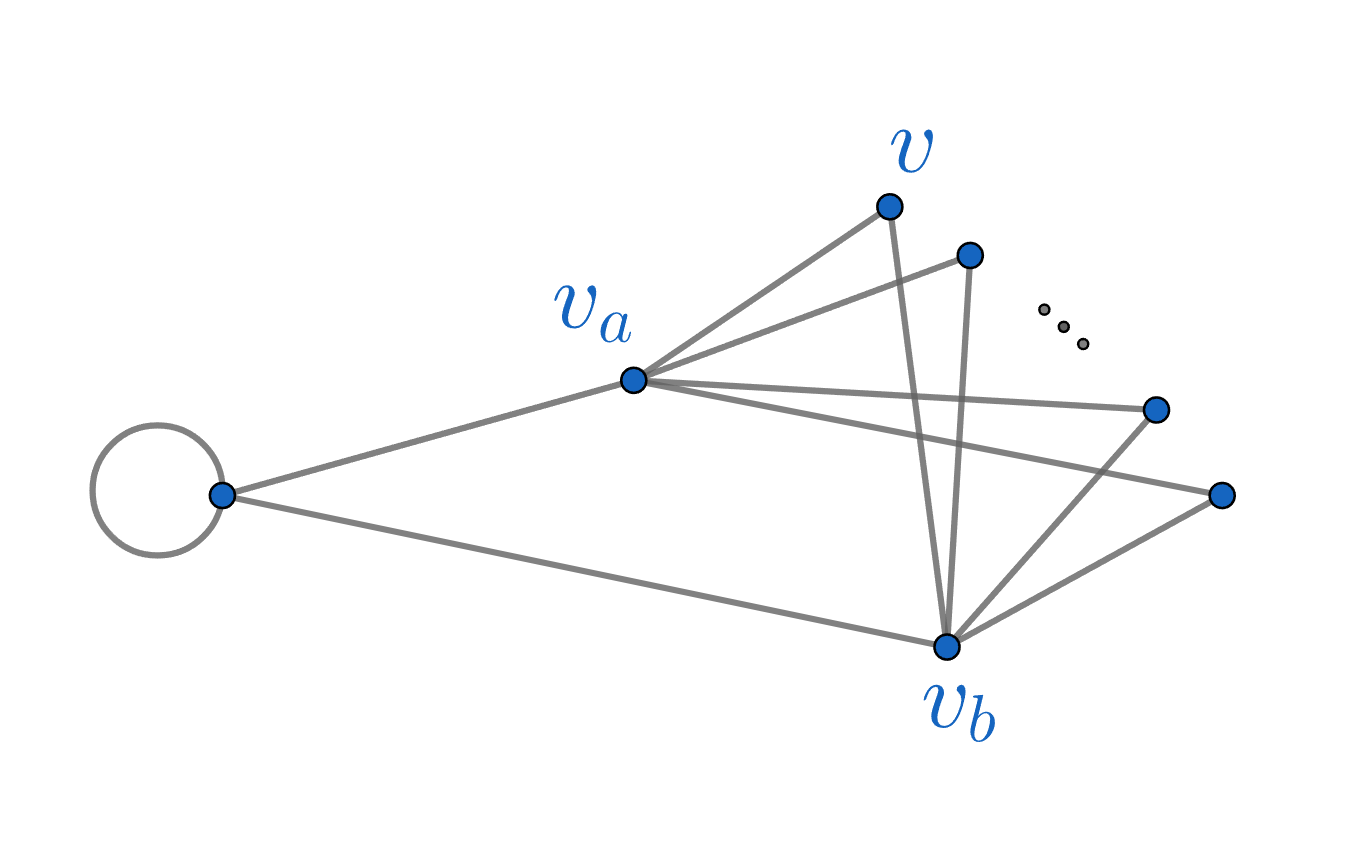}
			\caption{$H_1$}
			\label{fig:H1}
		\end{minipage}
		\hspace{3em}
		\begin{minipage}[b]{0.4\textwidth}
			\centering    
			\includegraphics[width=0.9\textwidth]{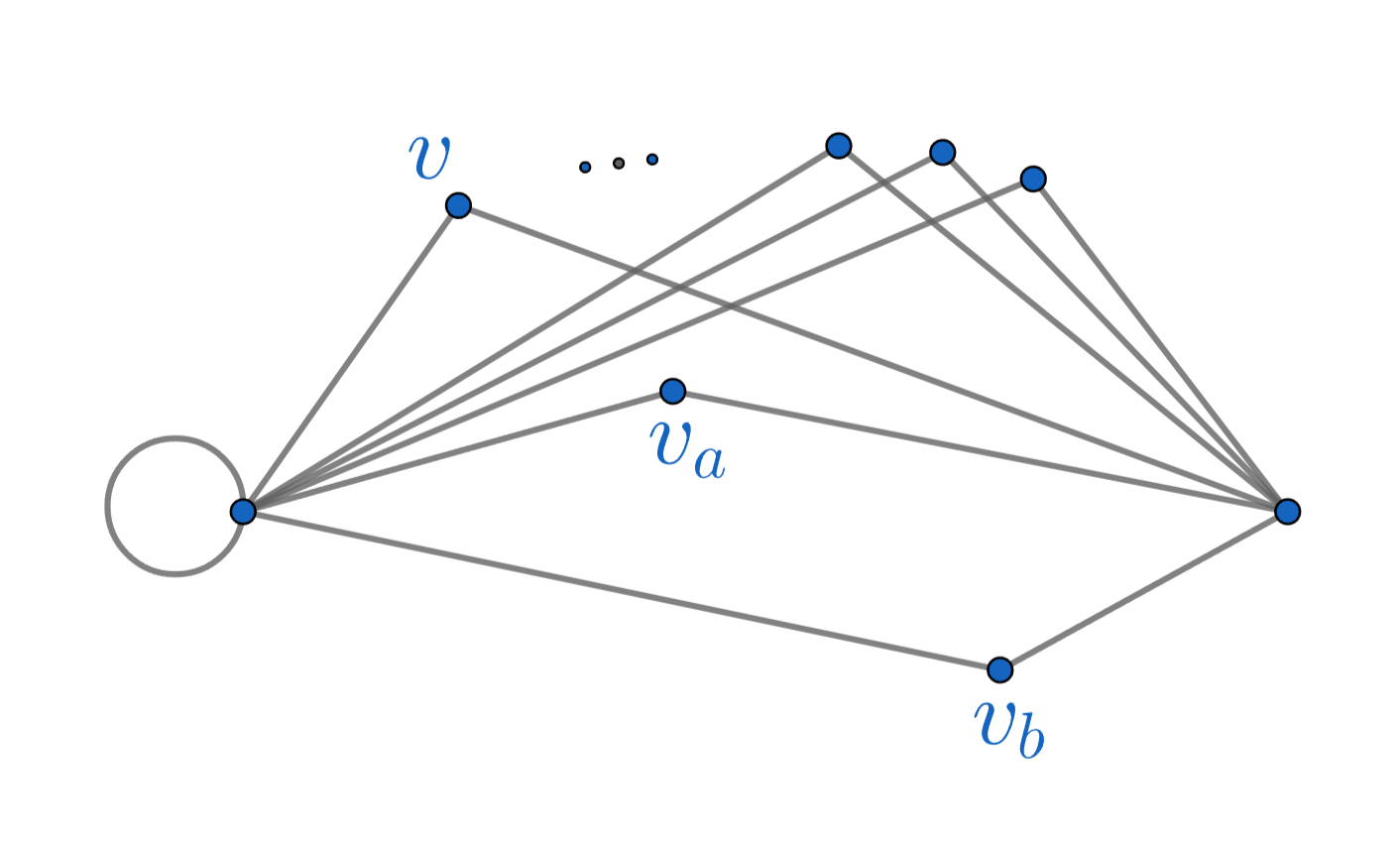}
			\caption{$H_2$}
			\label{fig:H2}
		\end{minipage}
	\end{figure}
	
	\vspace{4em}
	Note that the edge sets $E(H_1)$ and $E(H_2)$ are maximal with respect to rank 3. For if there is an edge $e' \notin E(H_i), i=1,2,$ such that $e'$ is incident with any two vertices in $V(H_i),$ then either at least a triangle is formed, or the new graph contains a subgraph of rank 4, a contradiction. Thus, $H_1$ and $H_2$ are the only triangle-free connected self-loop graphs of rank 3 when adding independent loopless vertices.
\end{proof}

\begin{remark}
	\label{rem1}
	It can be verified that the only self-loop 4-cycle $(C_4)_S$ with $2 \leq |S| \leq 4$ of rank 3 is $(C_4)_S$ with $|S|=2.$ Moreover, such $S$ must be an independent set in $C_4.$ For any other cases, $(C_4)_S$  has rank 4.
\end{remark}

\begin{theorem}	
	Up to isomorphisms, all rank 3 triangle-free connected cyclic graphs with two self-loops constructed from  $G_S=(C_4)_S$ with $|S|=2$ are  
	\[
	H'_2= G_S \underset{V(G_S) \setminus \cV_I}{\vee} W, 
	\]	
	where  $\cV_{I}$ is the independent set of loopless vertices in $V(G_S),$ and $W \nsubseteq V(G_S)$ is an independent set of $K_1'$s. 
	
\end{theorem}

\begin{proof}
	By Remark~\ref{rem1}, it suffices to consider only the graph $(C_4)_S$ with $\sigma=2$ and $S$ is an independent set in $C_4$. 
	The approach is similar to that of Theorem~\ref{thm1}, there are three adjacency cases (i), (ii), and (iii). Suppose $W$ is an independent set of isolated vertices possibly with loops. Then, up to isomorphisms,
	
	\begin{enumerate}[(i)]
		\item There are two sub-cases, each corresponds to the graphs below, respectively. 
		\begin{figure}[H]
			\centering
			\vspace{-1em}
			\begin{minipage}[b]{0.4\textwidth}
				\centering    
				\includegraphics[width=0.55\textwidth]{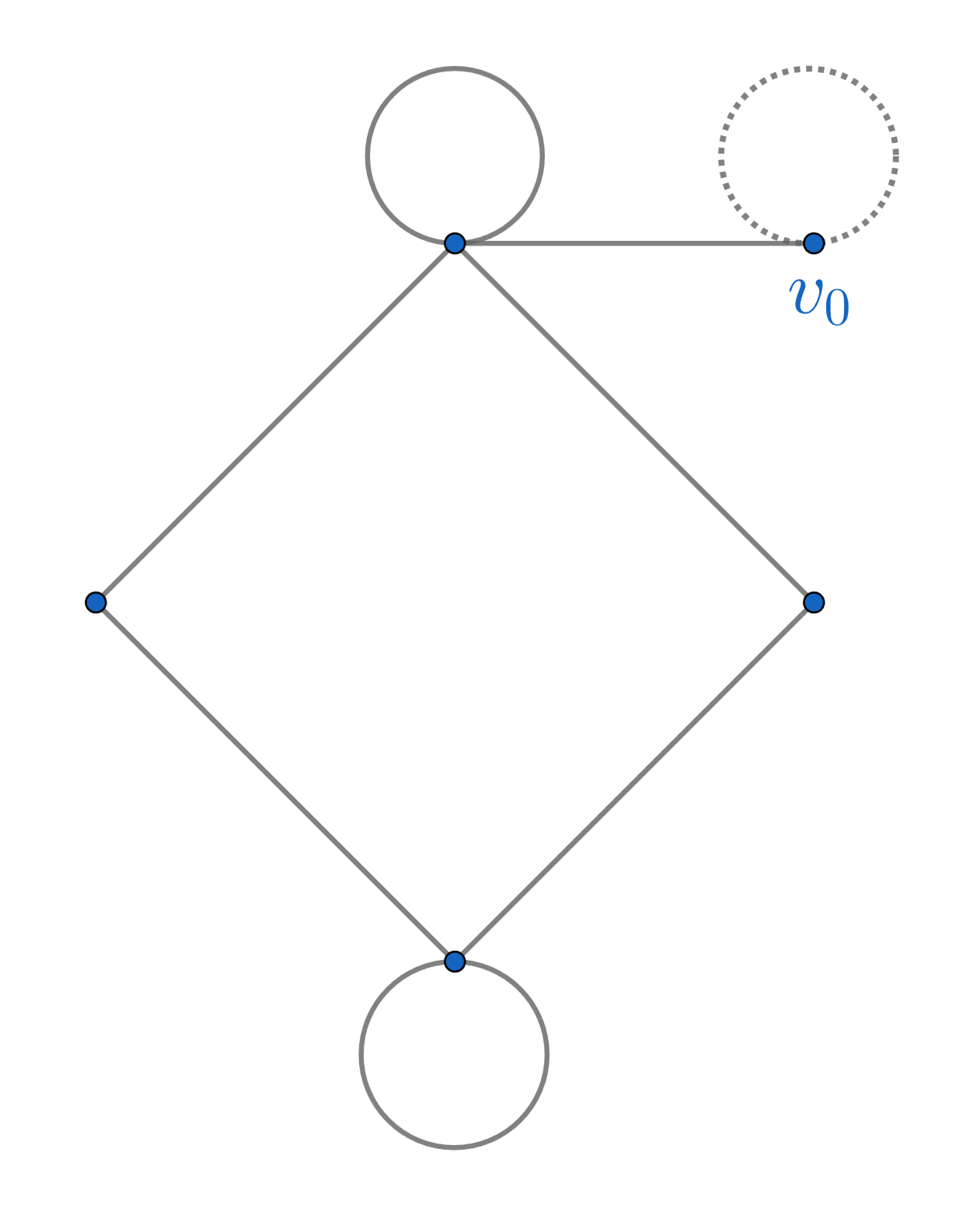}
				\caption{}
			\end{minipage}
			\begin{minipage}[b]{0.4\textwidth}
				\centering    
				\includegraphics[width=0.80\textwidth]{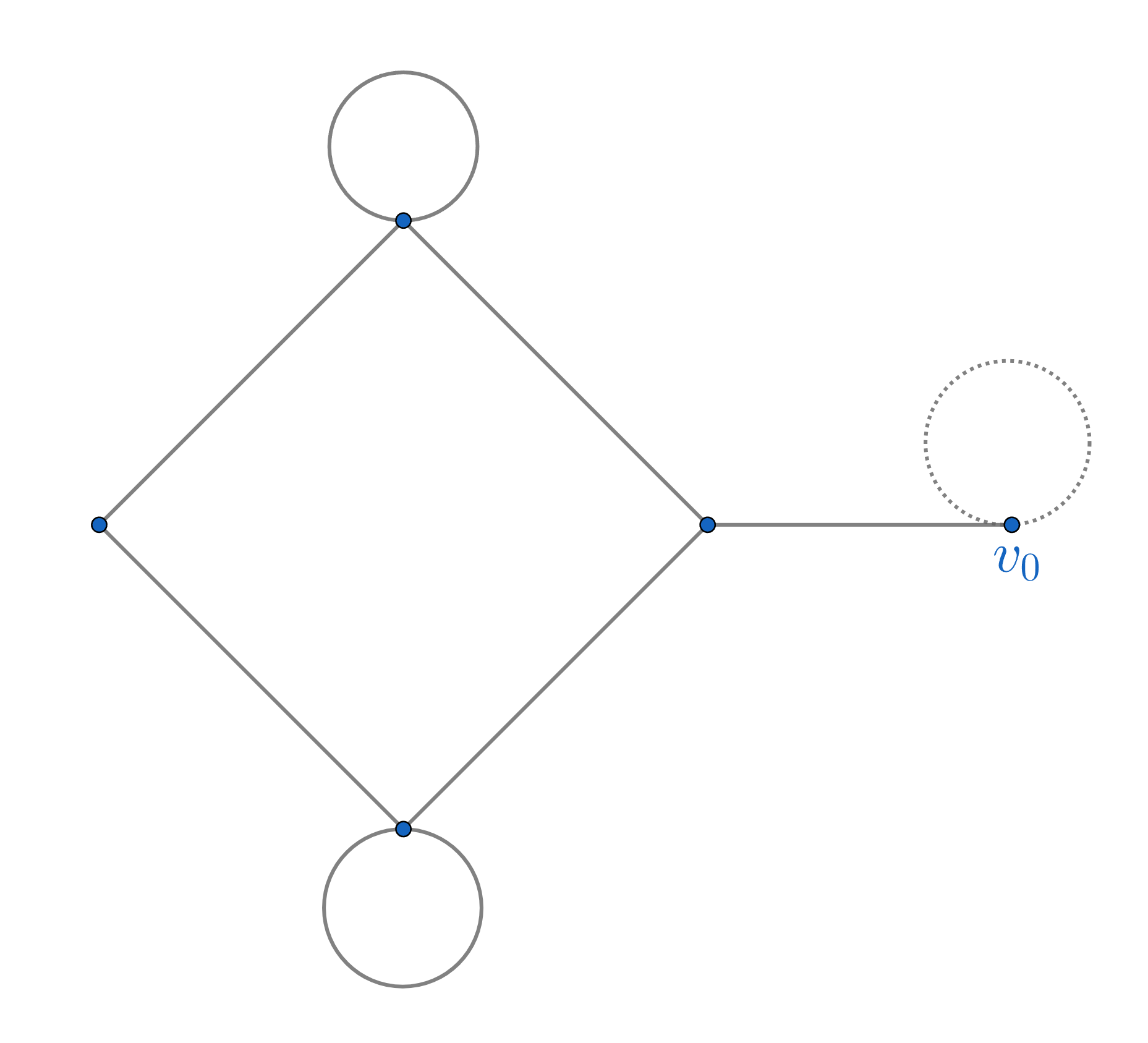}
				\caption{}
			\end{minipage}
		\end{figure}		
		\noindent However, regardless of whether $v_0$ has a loop or not, the left is of rank 4 and the right is of rank 5. 
		\item In this case, the corresponding graph is Fig.~\ref{fig:20}, which has rank 4, regardless of whether $v_0$ has a loop or not.
		\item The last case corresponds to the graph in Fig.~\ref{fig:21}, which has rank 3 if $v_0$ has no loop and rank 4 otherwise.
	\end{enumerate}
	
	\begin{figure}[H]
		\centering
		\begin{minipage}[b]{0.4\textwidth}
			\centering    
			\includegraphics[width=0.6\textwidth]{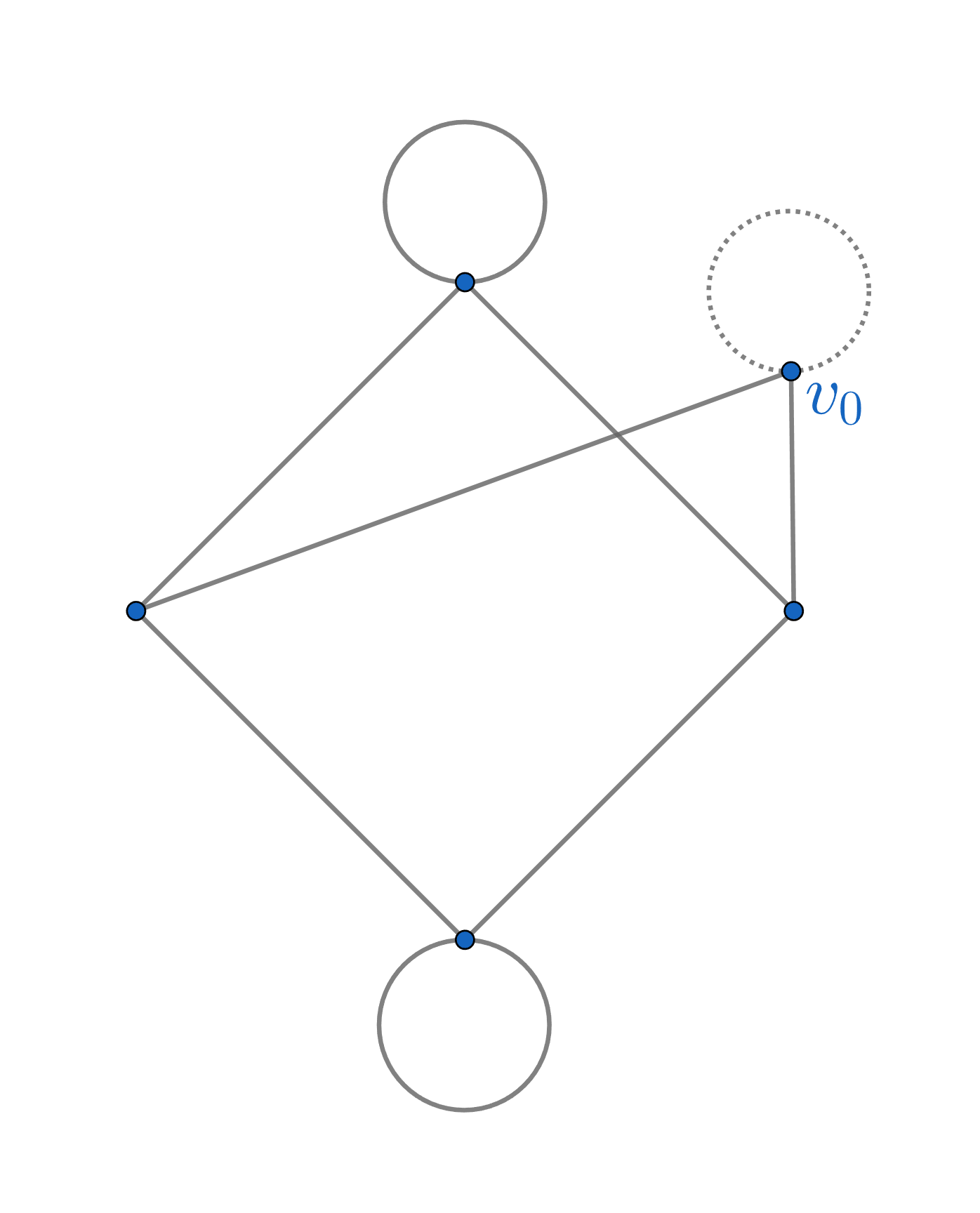}
			\caption{}
			\label{fig:20}
		\end{minipage}
		\begin{minipage}[b]{0.4\textwidth}
			\centering    
			\includegraphics[width=0.6\textwidth]{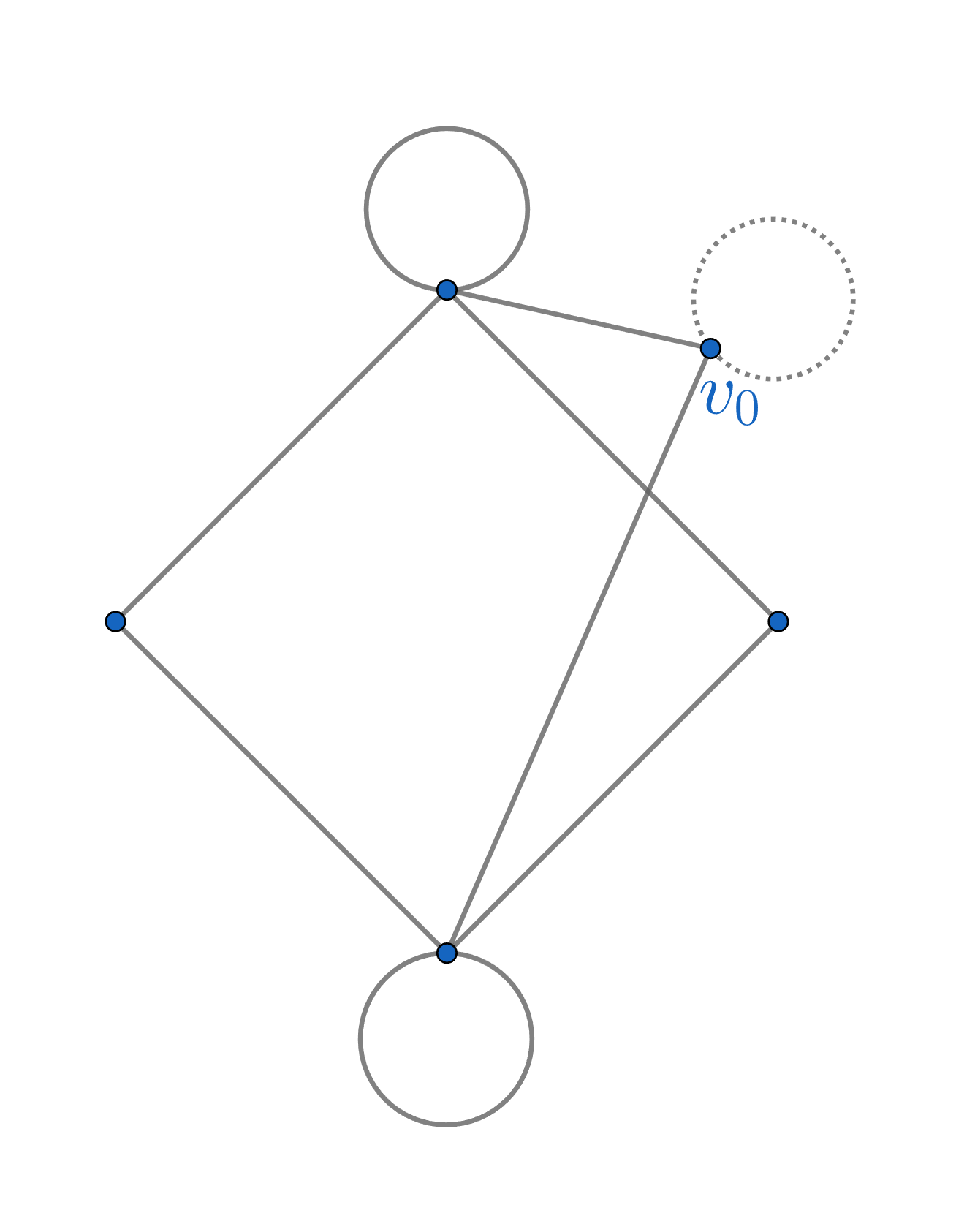}
			\caption{}				
			\label{fig:21}
		\end{minipage}
	\end{figure}
	Thus, only Case (iii) is possible to obtain rank 3 graphs and $W$ must contain only loopless vertices. By induction on graph joins with $W,$ the graph $H_2'$ (Fig.~\ref{fig:H2'}) is of rank 3 and two loops. The maximality of $E(H_2')$ is similar to the proof of Theorem~\ref{thm1}.
\end{proof}
\begin{figure}[h]
	\centering
	\includegraphics[width=0.42\textwidth]{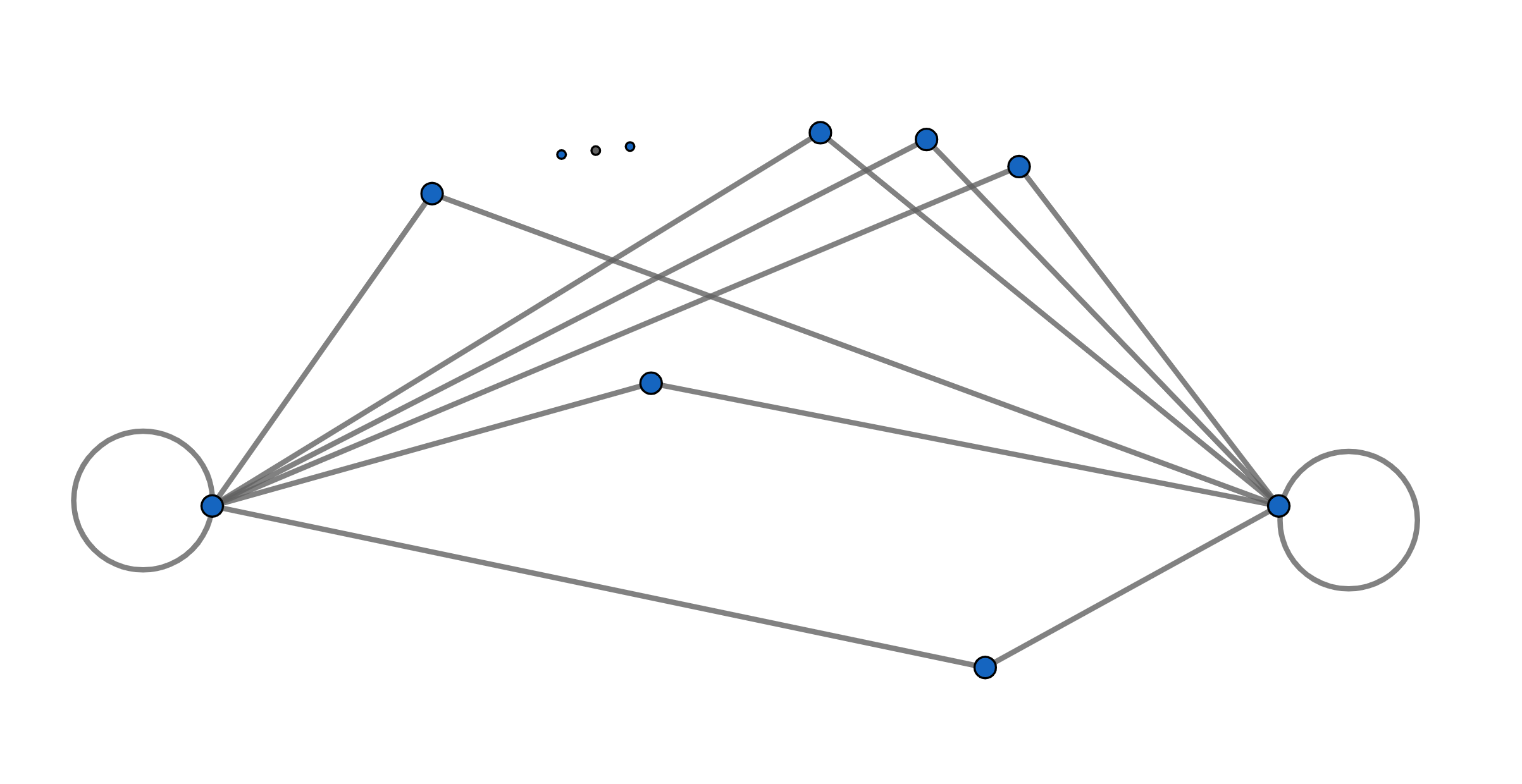}
	\caption{$H_2'$} \qedhere 
	\label{fig:H2'}
\end{figure} 

\begin{remark}
	Observe that the graph in Fig.~\ref{fig:23} (resp. \ref{fig:24}) obtained by adding a loop at $v$ in Fig~\ref{fig:14} (resp. \ref{fig:15}) are both of rank 4. If $W$ is an independent set of isolated vertices possibly with loops, then one verifies that the $n$-graph joins of Fig.~\ref{fig:23} and Fig.~\ref{fig:24} with $W$ cannot be of rank 3 if $W$ contains at least one $\widehat{K_1}.$ However, removing such looped vertex 
	from $W$ in the graph joins may create new classes of triangle-free connected self-loop cyclic graphs of rank 3, as shown in the next theorem.
	\begin{figure}[H]
		\centering
		\hspace{-2em}
		\begin{minipage}[b]{0.33\textwidth} 
			\centering    
			\includegraphics[width=1\textwidth]{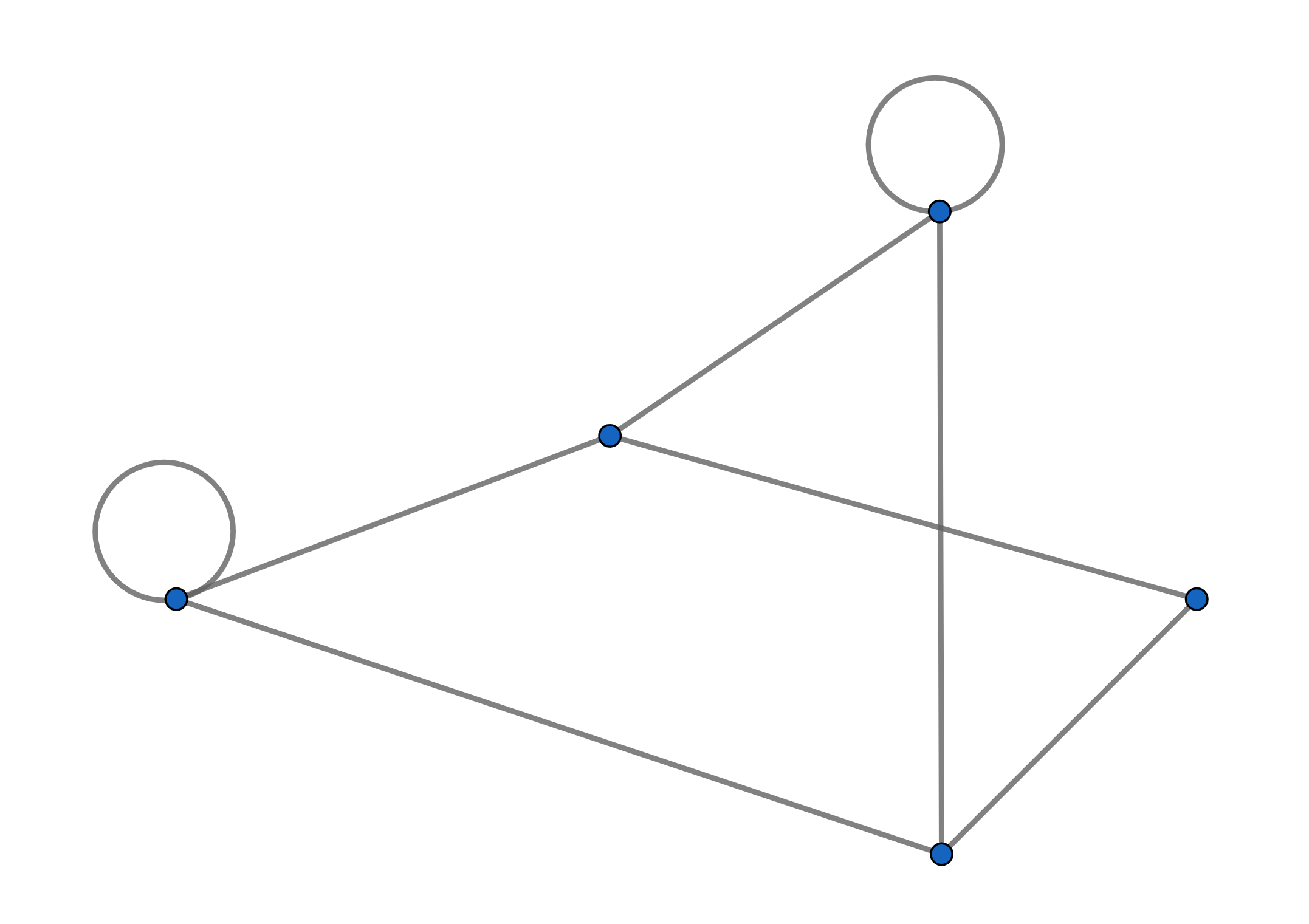}
			\caption{ }
			\label{fig:23}
		\end{minipage}
		\hspace{4em}
		\begin{minipage}[b]{0.33\textwidth} 
			\centering    
			\includegraphics[width=1.1\textwidth]{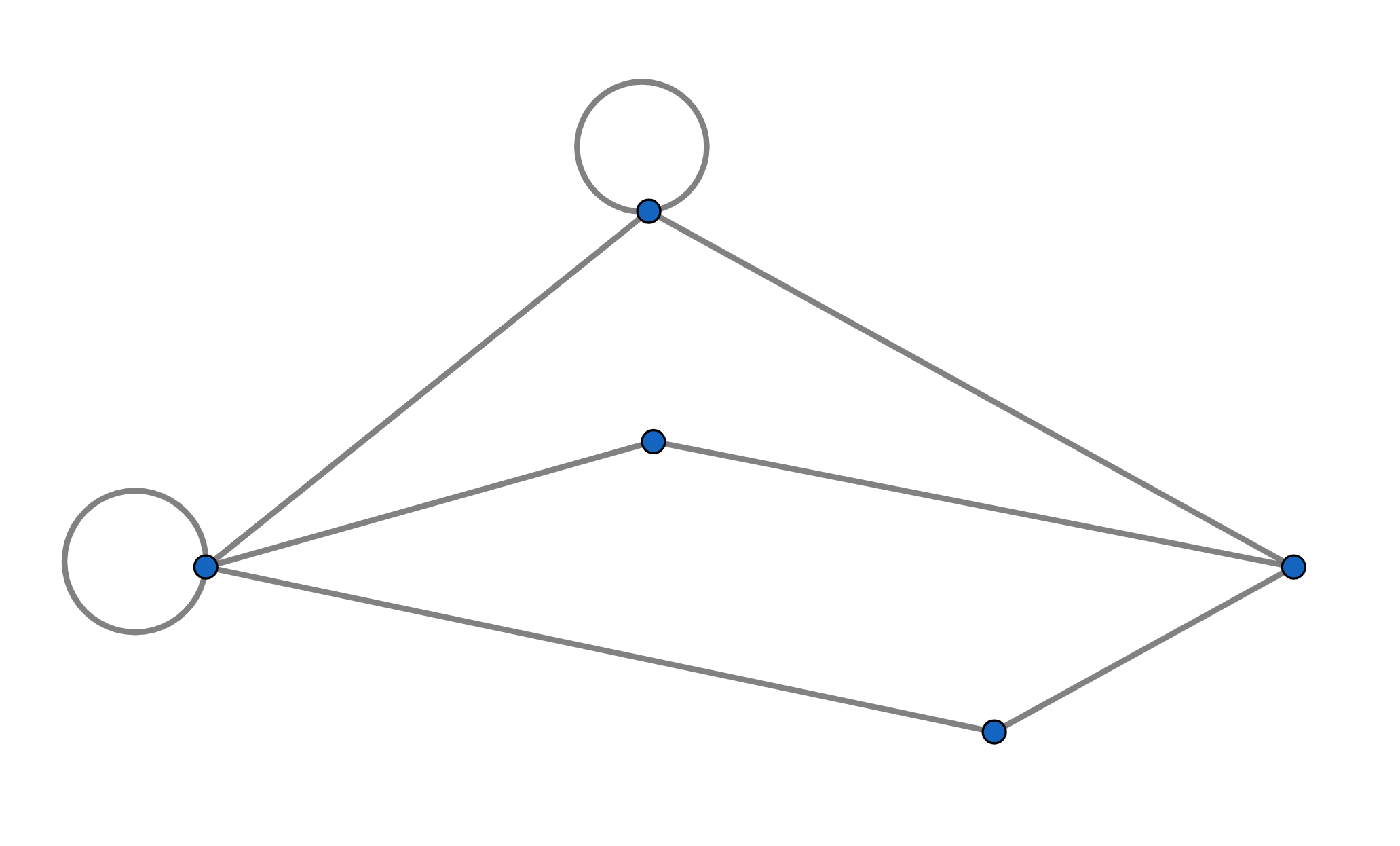}
			\caption{ }
			\label{fig:24}
		\end{minipage}
	\end{figure}
\end{remark}

\begin{theorem}
	Let $F_S$ and $F_S'$ be the two graphs obtained from $G_S=(C_4)_S,$ $|S|=1,$ 
	by adding an one-sided looped pendent edge as in Fig.~\ref{fig:8} and Fig.~\ref{fig:9}, respectively. Then, 
	up to isomorphisms, all rank 3 triangle-free connected cyclic graphs with two self-loops constructed from $F_S$ and $F_S'$ are
	
	\[
	H_3 = F_S \underset{\cV_I}{\vee} W,  \qquad 
	H_4 = F_S \underset{V(G_S) \setminus \cV_I}{\vee} W, \qquad  
	H_5 = F'_S \underset{\cV_I}{\vee} W,
	\]
	where $\cV_{I}$ is the independent set of $V(G_S)\setminus S$ and $W \nsubseteq V(G_S)\cup \{v_0\}$  is an independent set of $K_1'$s.

	\vspace{-1em}    
	\begin{figure}[H]
		\centering
		\hspace{-5em} 
		\begin{minipage}[b]{0.4\textwidth} 
			\centering    
			\includegraphics[width=0.6\textwidth]{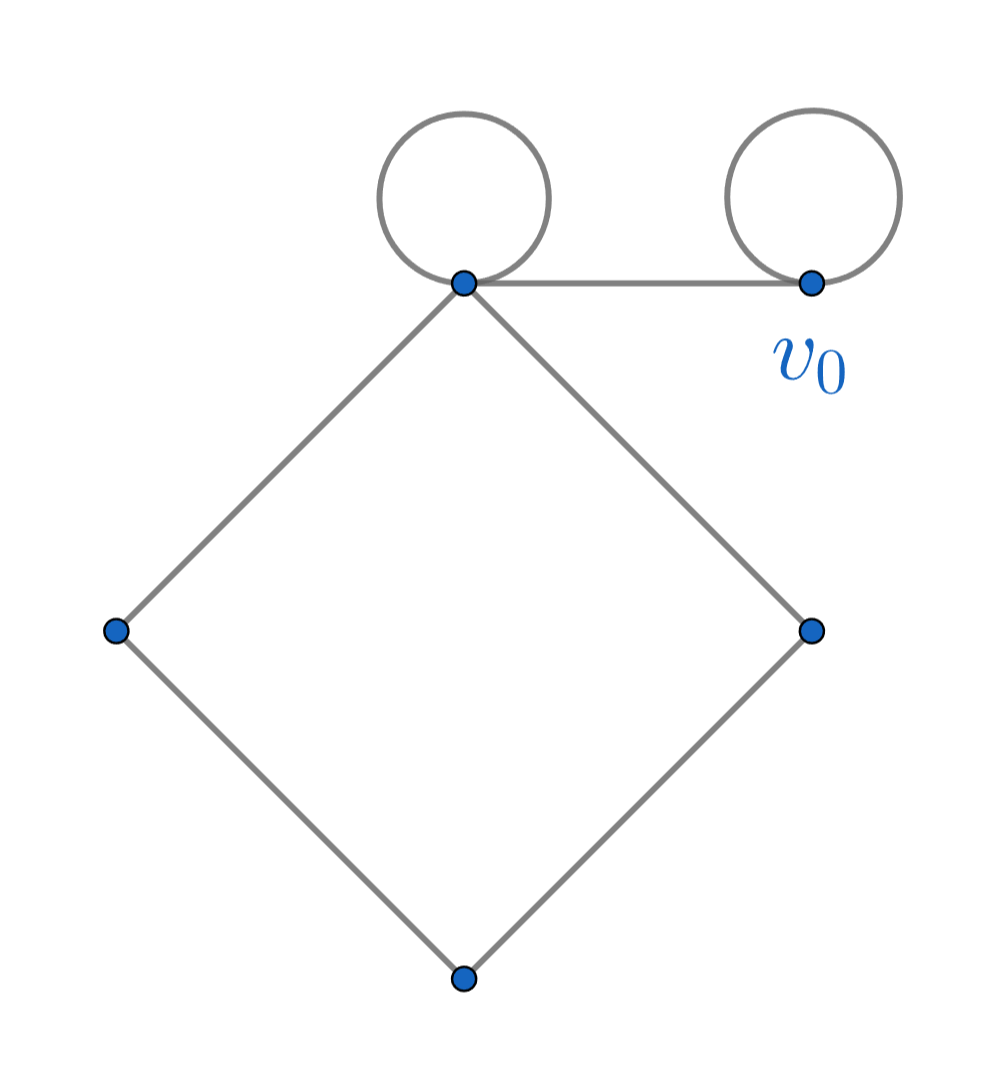}
			\caption{$F_S$}
			\label{fig:8}
		\end{minipage}
		\hspace{1em}
		\begin{minipage}[b]{0.4\textwidth} 
			\centering    
			\includegraphics[width=0.65\textwidth]{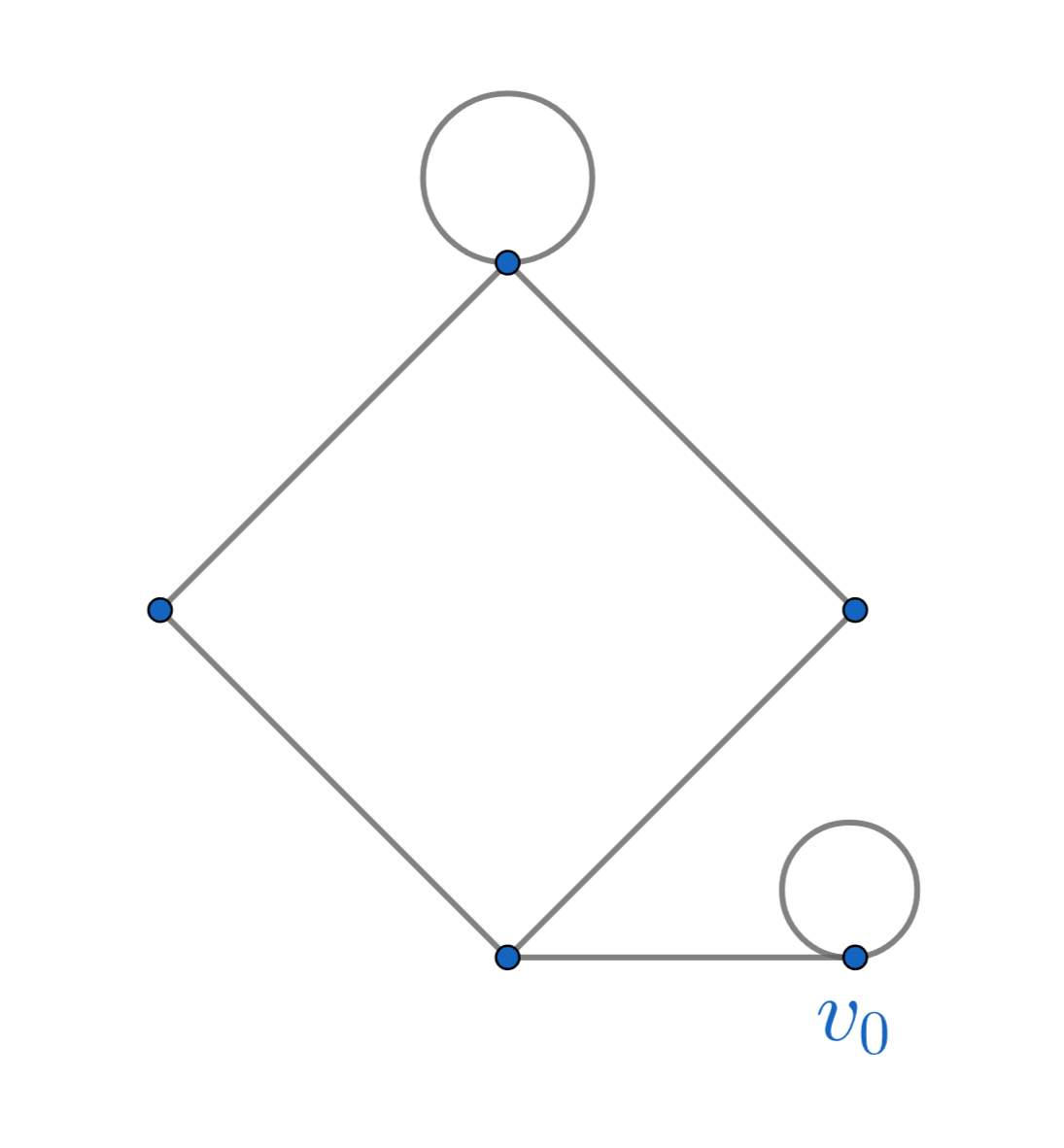}
			\caption{$F'_S$}
			\label{fig:9}
		\end{minipage}
	\end{figure}
\end{theorem}

\begin{proof}
	Let $\cV_{I}=\{v_a,v_b\}$ be the independent set in $V(G_S) \setminus S.$ If $v_0=\widehat{K_1}$ is adjacent to all vertices in $V(G_S) \setminus \cV_I,$ then the resulting graph is of rank 4. Thus, $v_0$ is only adjacent to either vertex in $V(G_S) \setminus \cV_I.$ 
	This produces $F_S$ and $F_S'$ of rank 3. Observe that adding a distinct vertex (with or without a loop) to $F_S$ with one edge will produce a graph that contains a subgraph of rank at least 4. As such, the only possible rank 3 graphs of desired are as follows in Fig.~\ref{fig:18} and Fig.~\ref{fig:19}
	\begin{figure}[H]
		\centering
		\begin{minipage}[b]{0.4\textwidth} 
			\centering    
			\includegraphics[width=0.8\textwidth]{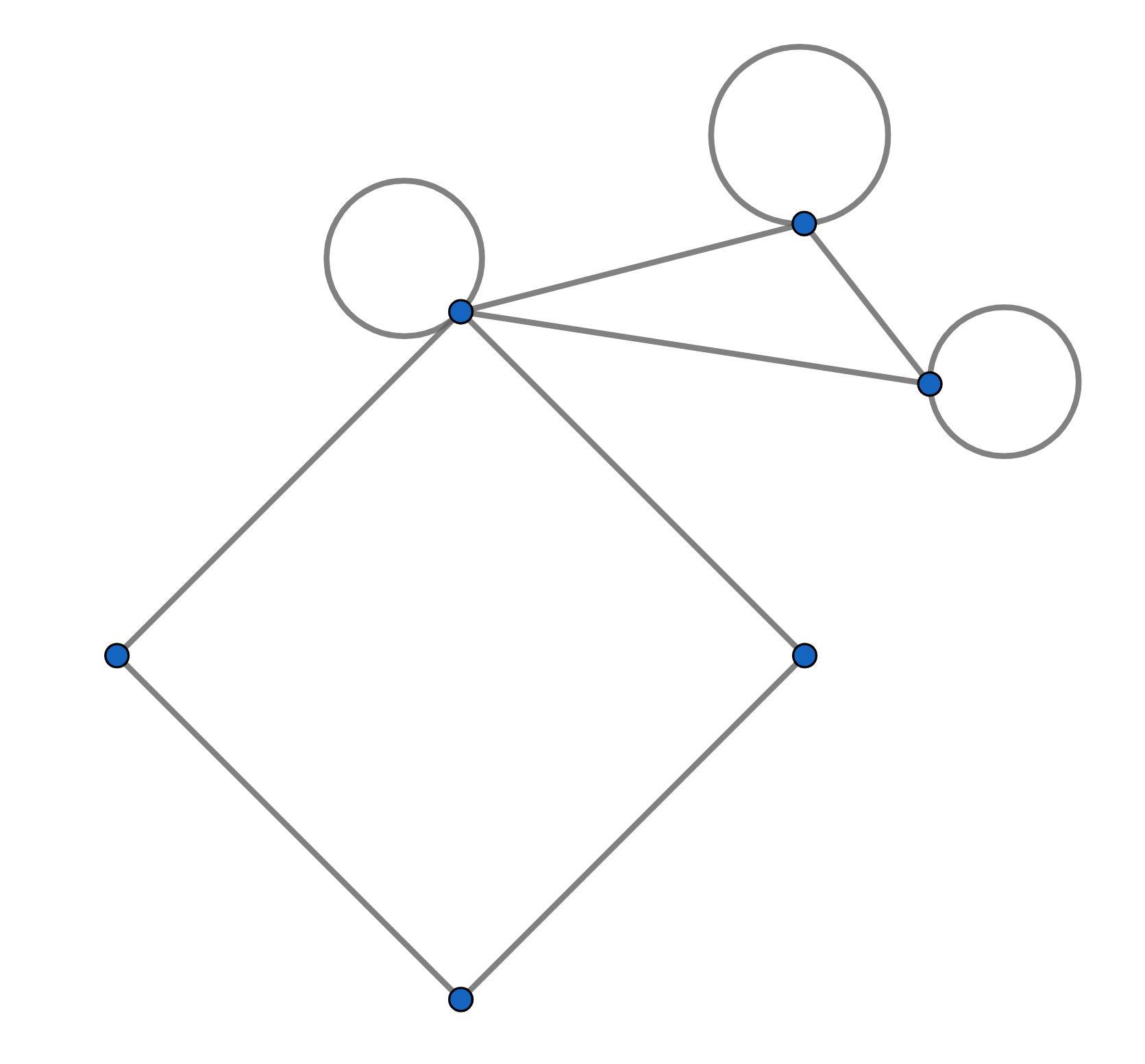}
			\caption{ }
			\label{fig:18}
		\end{minipage}
		\hspace{1em}
		\begin{minipage}[b]{0.4\textwidth} 
			\centering    
			\includegraphics[width=0.8\textwidth]{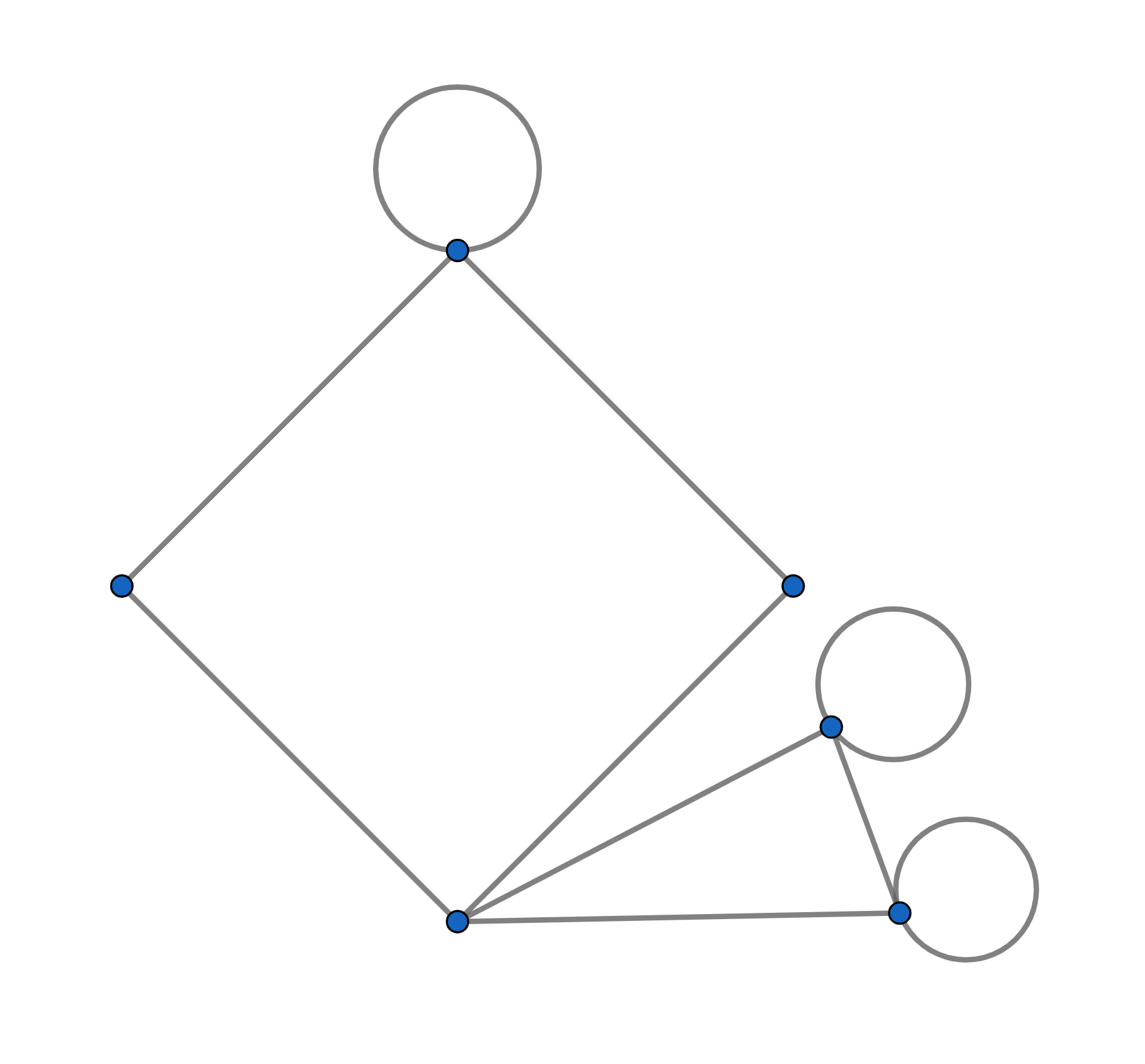}
			\caption{ }
			\label{fig:19}
		\end{minipage}
	\end{figure}
	\noindent which, however, clearly contain a triangle. Thus, the added vertices must be loopless. By Corollary~\ref{subgraphrank} and applying graph joins, the only possible desired graphs are $H_3, H_4,$ and $H_5,$ respectively:
	\begin{figure}[H]
		\centering
		\hspace{-2em}
		\begin{minipage}[b]{0.35\textwidth} 
			\centering    
			\includegraphics[width=0.85\textwidth]{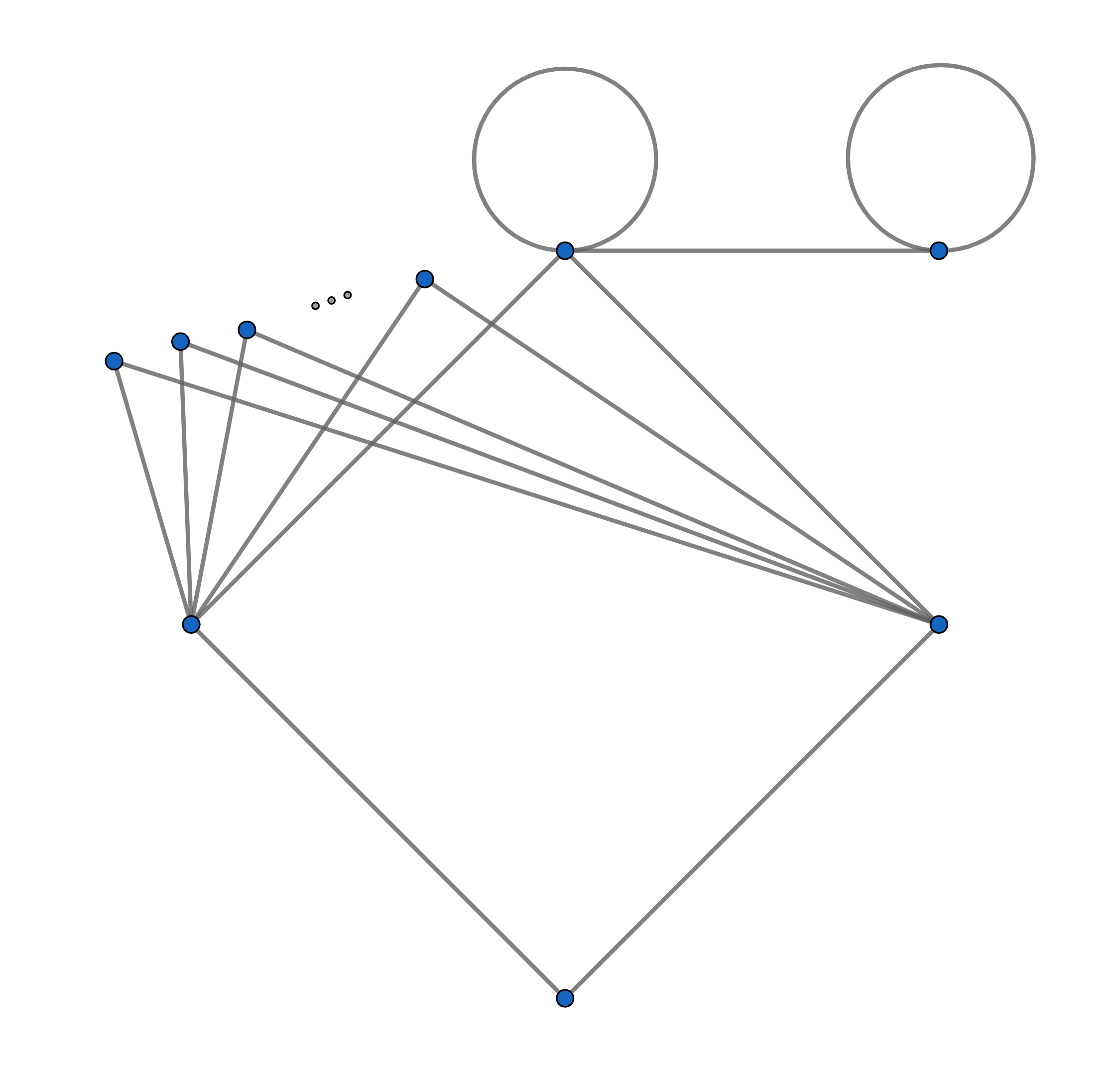}
			\caption{$H_3$}
			\label{fig:H3}
		\end{minipage}
		\hspace{-2em}
		\begin{minipage}[b]{0.35\textwidth} 
			\centering    
			\includegraphics[width=0.9\textwidth]{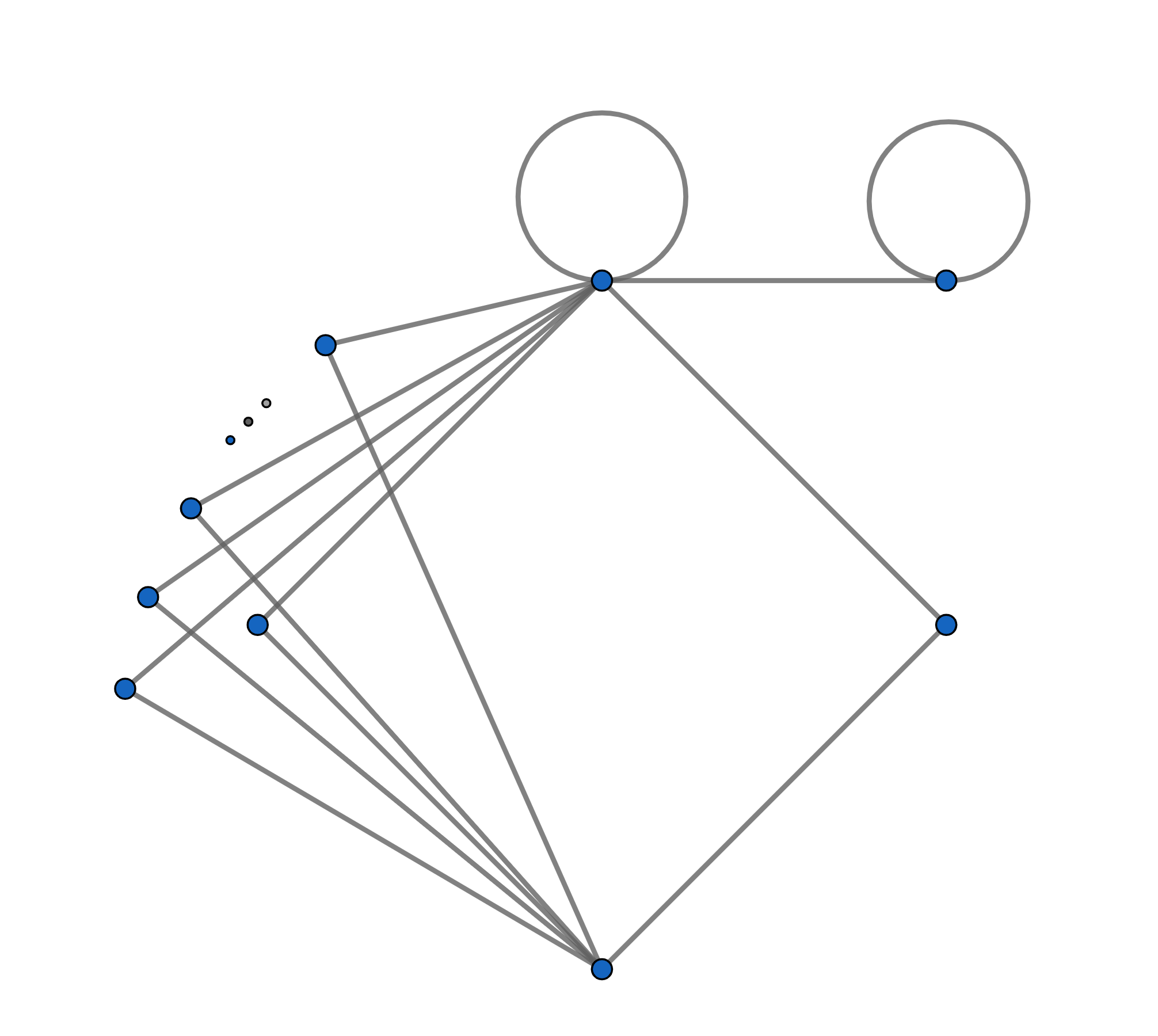}
			\caption{$H_4$}
			\label{fig:H4}
		\end{minipage}
		\hspace{-2em}
		\begin{minipage}[b]{0.35\textwidth} 
			\centering    
			\includegraphics[width=0.9\textwidth]{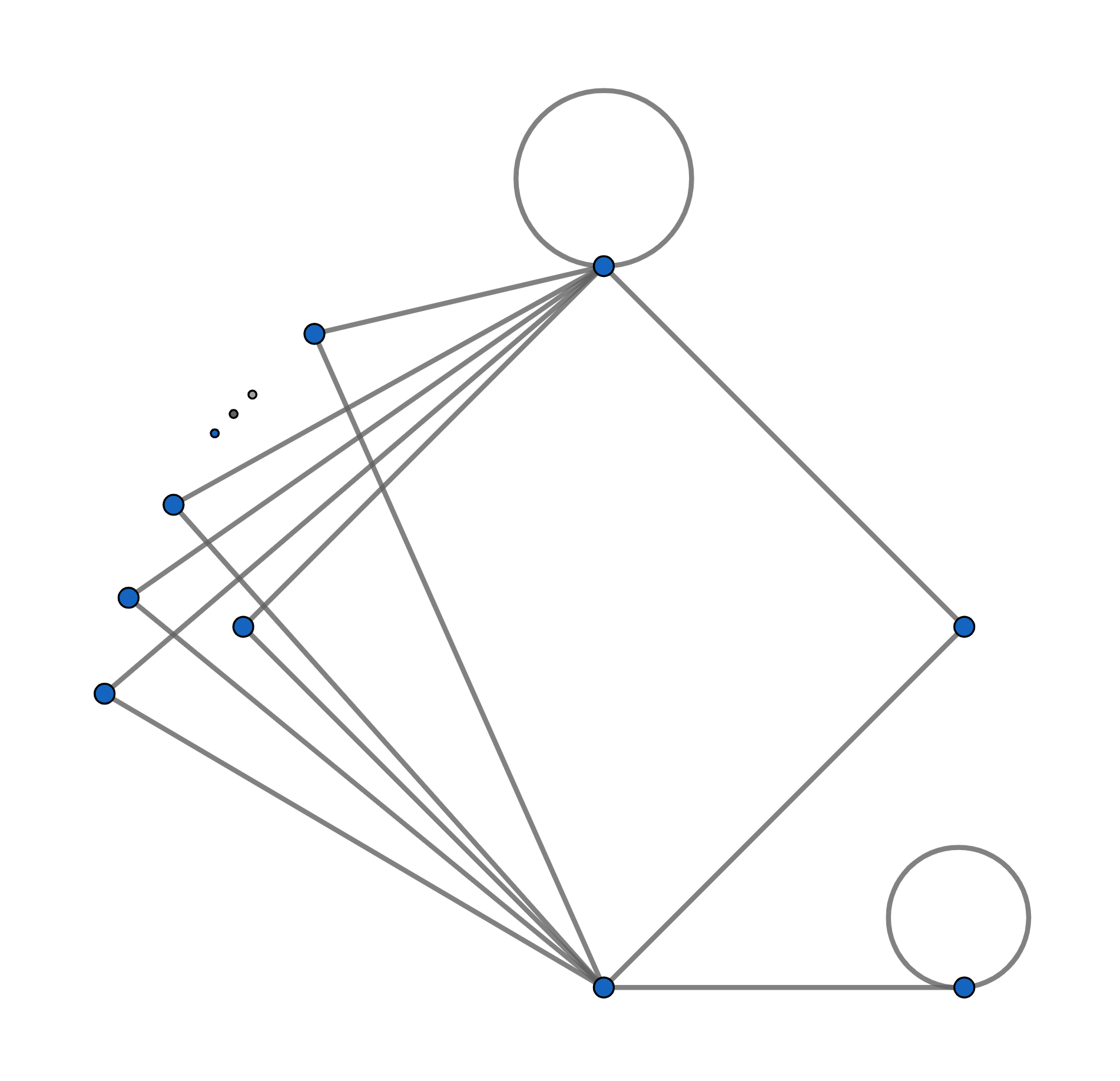}
			\caption{$H_5$}
			\label{fig:H5}
		\end{minipage} \qedhere
	\end{figure}
\end{proof}

\begin{remark}
	Let $H_3', H_4',$ and $H_5'$ be the corresponding graphs with $V(G_S) \setminus \cV_I \subseteq S.$ Then, all these graphs have rank 4:
	\begin{figure}[H]
		\centering
		\begin{minipage}[b]{0.35\textwidth} 
			\centering    
			\vspace{-2em}   
			\includegraphics[width=0.85\textwidth]{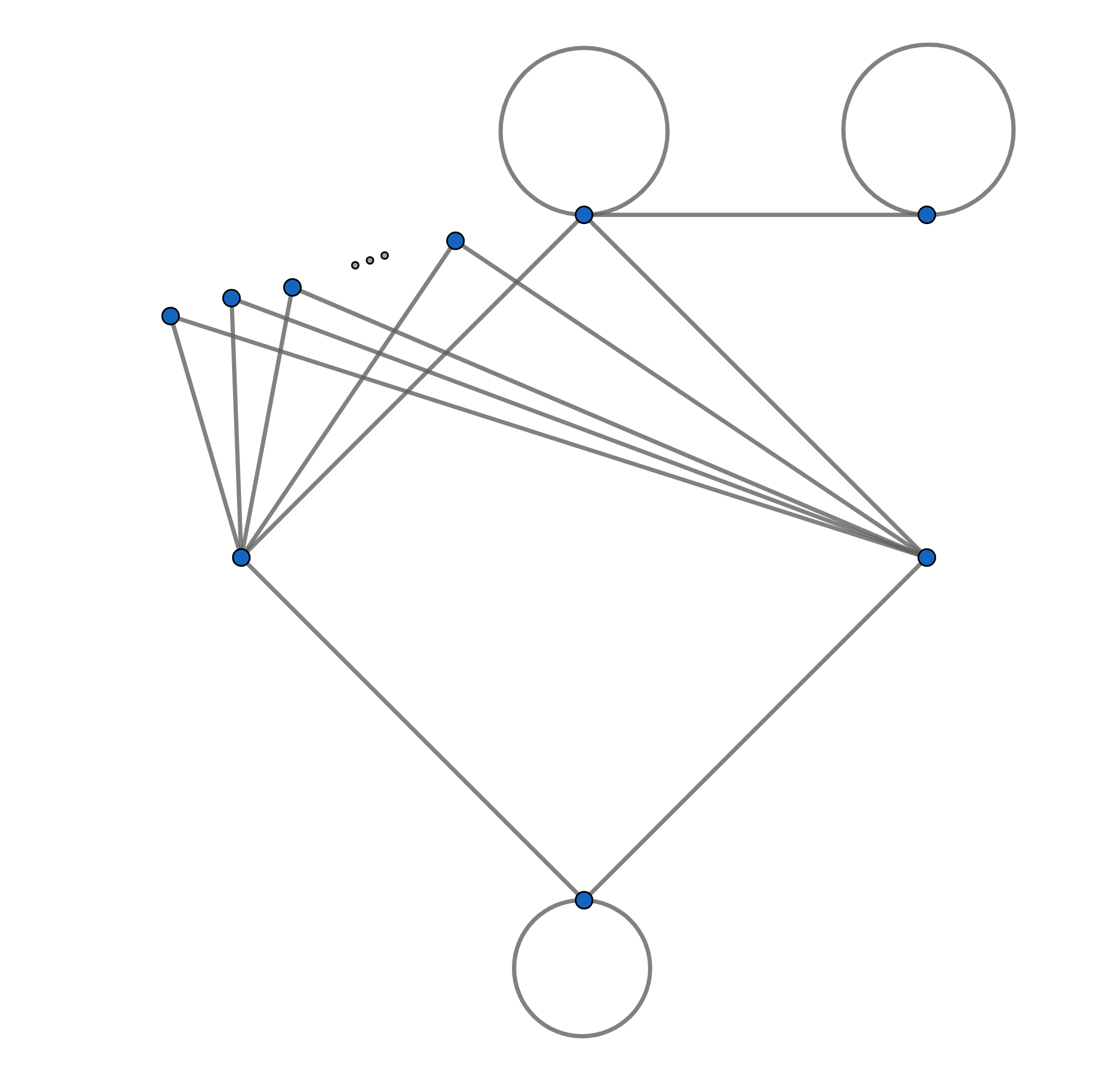}
			\caption{$H_3'$}
			\label{fig:32}
		\end{minipage}
		\hspace{-2em}
		\begin{minipage}[b]{0.35\textwidth} 
			\centering    
			\includegraphics[width=0.9\textwidth]{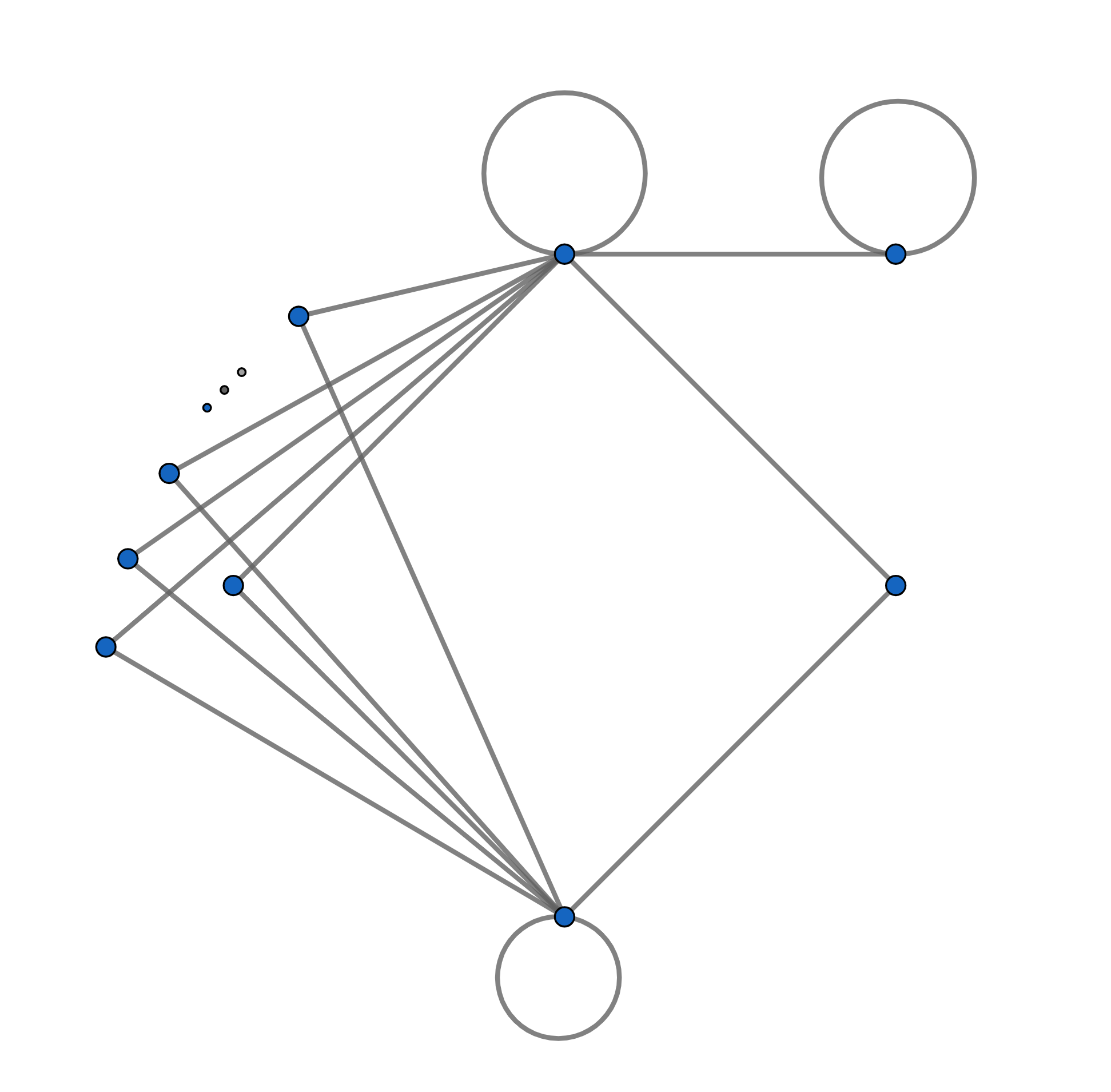}
			\caption{$H_4'$}
			\label{fig:33}
		\end{minipage}
		\hspace{-2em}
		\begin{minipage}[b]{0.35\textwidth} 
			\centering    
			\includegraphics[width=0.9\textwidth]{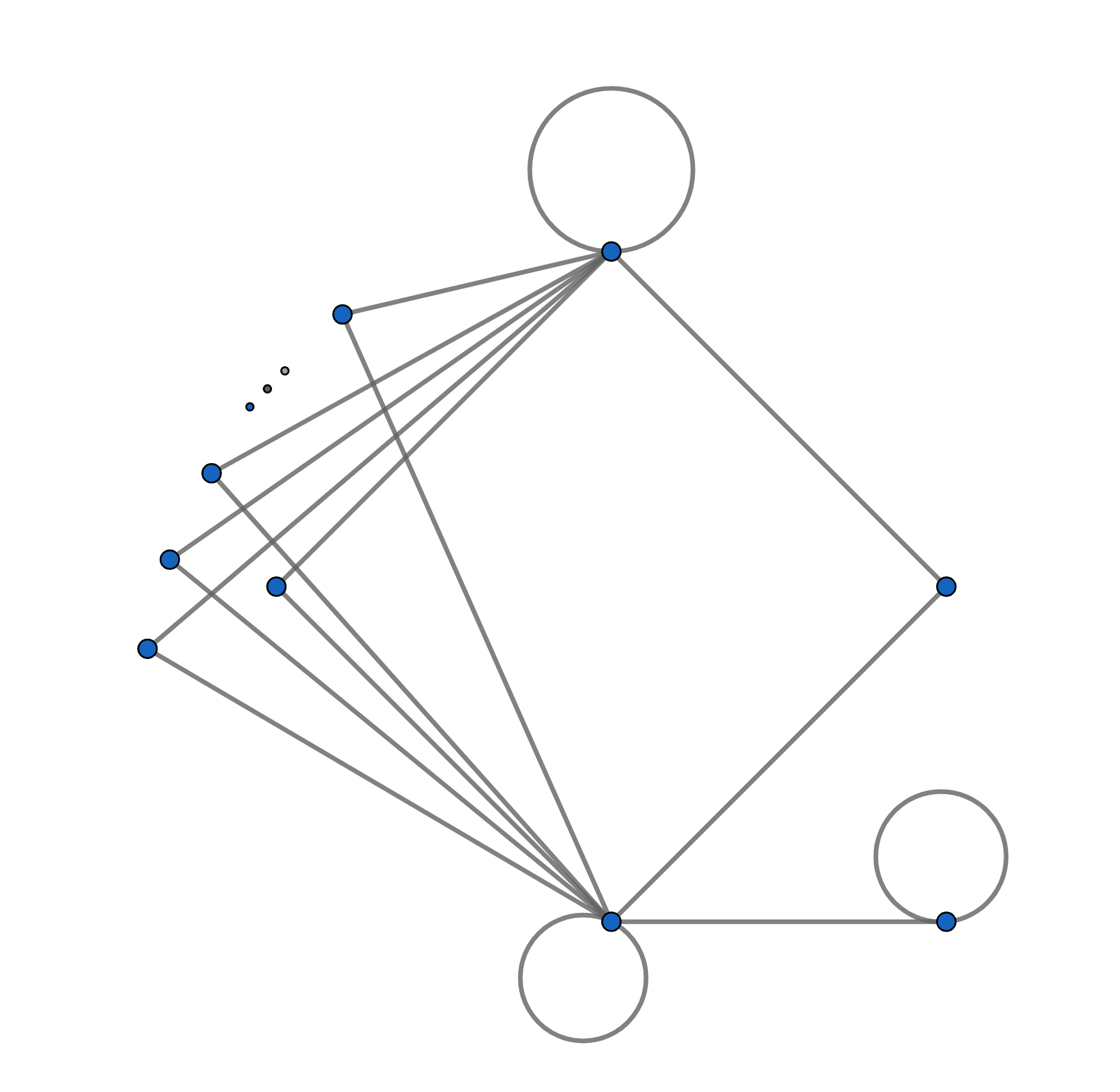}
			\caption{$H_5'$}
			\label{fig:34}
		\end{minipage}
		\vspace{-1em}
	\end{figure}
	\noindent where, of course, $H_4'$ is isomorphic to $H_5'.$
\end{remark}

\section{Conclusion}
\label{conclusion}

We have shown that if a triangle-free connected self-loop cyclic graph is not isomorphic to either of the six graphs in Fig.~\ref{fig:H1}, Fig.~\ref{fig:H2}, Fig.~\ref{fig:H2'}, Fig.~\ref{fig:H3}, Fig.~\ref{fig:H4}, and Fig.~\ref{fig:H5}, then such graph cannot be of rank 3. The necessary condition of the characterization is then obtained by the contrapositive argument. A sufficient condition is straightforward by checking the rank of their adjacency matrix: if a triangle-free connected self-loop cyclic graph is isomorphic to either of the six graphs aforementioned, then it is of rank 3. Thus, we have provided the desired characterization and a partial solution to the complete characterizations of rank 3 connected self-loop graphs.

\section{Open problem}
\label{openproblem}

\begin{enumerate}[(i)]
	\item It is evident from Fig.~\ref{fig:18} and Fig.~\ref{fig:19} there are more rank 3 connected self-loop graphs of order at least 4 that \textit{contain a triangle}. Thus, it would be interesting to characterize all such graphs.
	\item One can also investigate the relationship between the rank and energy of self-loop graphs. 
\end{enumerate}

\vspace{0.5cm}
\subsection*{Acknowledgement}
The author would like to thank the anonymous referees, Miin Huey Ang, and Kevin Fung for their helpful comments and suggestions in improving the manuscript.


\bibliography{bibliography}{}
\bibliographystyle{amsplain}



\end{document}